\documentclass[table]{amsart} 
\usepackage{amsmath}
\usepackage[margin=1.0in]{geometry}
\usepackage{amscd,amsmath,amsxtra,amsthm,amssymb,stmaryrd,xr,mathrsfs,mathtools,enumerate,commath, comment, enumitem, graphicx}
\usepackage{amsfonts, amssymb, amsthm, amsmath, braket, hyperref, mathtools, comment, mathrsfs, enumitem, cleveref}
\usepackage{stmaryrd}
\usepackage{orcidlink}
\usepackage{multirow}
\usepackage{xcolor}
\usepackage{commath}
\usepackage{comment}
\usepackage{tikz-cd}
\usepackage{tkz-graph}
\usepackage{colonequals}
\usepackage{hyperref}
\usepackage{graphicx}
\usepackage{bigstrut}
\usepackage{longtable, multirow, array}

\usepackage{longtable} 
\usepackage{pdflscape} 
\usepackage{booktabs}
\usepackage{hyperref}
\definecolor{vegasgold}{rgb}{0.77, 0.7, 0.35}
\definecolor{darkgoldenrod}{rgb}{0.72, 0.53, 0.04}
\definecolor{gold(metallic)}{rgb}{0.83, 0.69, 0.22}
\hypersetup{
 colorlinks=true,
 linkcolor=darkgoldenrod,
 filecolor=brown,      
 urlcolor=gold(metallic),
 citecolor=darkgoldenrod,
 }
\newtheorem{lthm}{Theorem}

\usepackage[all,cmtip]{xy}
\DeclareFontFamily{U}{wncy}{}
\DeclareFontShape{U}{wncy}{m}{n}{<->wncyr10}{}
\DeclareSymbolFont{mcy}{U}{wncy}{m}{n}
\DeclareMathSymbol{\Sh}{\mathord}{mcy}{"58}
\usepackage[T2A,T1]{fontenc}
\usepackage[OT2,T1]{fontenc}

\usepackage{tikz}
\usetikzlibrary{shapes.geometric}
\usetikzlibrary{decorations.markings}
\tikzset{every loop/.style={min distance=10mm,looseness=10}}
\tikzstyle{vertex}=[auto=left,circle,minimum size=1pt,inner sep=0pt]

\newtheorem{theorem}{Theorem}[section]
\newtheorem{lemma}[theorem]{Lemma}
\newtheorem{ass}[theorem]{Assumption}
\newtheorem*{theorem*}{Theorem}
\newtheorem*{ass*}{Assumption}
\newtheorem{definition}[theorem]{Definition}

\newtheorem{remark}[theorem]{Remark}

\newtheorem{proposition}[theorem]{Proposition}

\newcommand{\Z}{\mathbb{Z}}

\newcommand{\Q}{\mathbb{Q}}

\newcommand{\cO}{\mathcal{O}}

\newcommand{\op}[1]{\operatorname{#1}}

\numberwithin{equation}{section}

\begin{document}

\title[Shapes of sextic pure number fields
]{On the distribution of shapes of sextic pure number fields
}

\author[A.~Jakhar]{Anuj Jakhar\, \orcidlink{0009-0007-5951-2261}}
\address[Jakhar]{Indian Institute Of Technology, Chennai, Tamil Nadu 600036, India}
\email{anujjakhar@iitm.ac.in}

\author[R.~Kalwaniya]{Ravi Kalwaniya\, \orcidlink{0009-0008-6964-5276}}
\address[Kalwaniya]{Indian Institute Of Technology, Chennai, Tamil Nadu 600036, India}
\email{ravikalwaniya3@gmail.com}

\author[A.~Ray]{Anwesh Ray\, \orcidlink{0000-0001-6946-1559}}
\address[Ray]{Chennai Mathematical Institute, H1, SIPCOT IT Park, Kelambakkam, Siruseri, Tamil Nadu 603103, India}
\email{anwesh@cmi.ac.in}

\author[B.~Roy]{Bidisha Roy\, \orcidlink{https://orcid.org/0000-0002-9143-2567}}
\address[Roy]{Indian Institute of Technology Tirupati, Yerpedu-Venkatagiri Road, Tirupati, Andhra Pradesh 517619, India}
\email{bidisha.roy@iittp.ac.in}

\subjclass[2020]{}

\maketitle

\begin{abstract}
The \emph{shape} of a number field $K$ of degree $n$ is defined as the equivalence class of the lattice of integers with respect to linear operations that are composites of rotations, reflections, and positive scalar dilations. The shape is a point in the space of shapes $\mathcal{S}_{n-1}$, which is the double quotient $\op{GL}_{n-1}(\mathbb{Z}) \backslash \op{GL}_{n-1}(\mathbb{R}) / \op{GO}_{n-1}(\mathbb{R})$. We investigate the distribution of shapes of pure sextic number fields
$K=\Q(\sqrt[6]{m})$, ordered by absolute discriminant.
Such fields are partitioned into $20$ distinct Types determined by local conditions at $2$ and $3$, and an explicit integral basis is given in each case.
For each Type, the shape of $K$ admits an explicit description in terms
of \emph{shape parameters}.
Fixing the sign of $m$ and a Type, we prove that the corresponding shapes
are equidistributed along a translated torus orbit in the space of shapes.
The limiting distribution is given by an explicit measure expressed as the product of a continuous measure and a discrete
measure.
\end{abstract}

\section{Introduction}
\subsection{Motivation and historical context}
The ring of integers $\cO_K$ of a number field $K$ of degree $n$ is a rank $n$ lattice whose geometry encodes precise arithmetic information about the number field. For instance, the square root of the absolute discriminant is the volume of its fundamental parallelepiped and encodes refined information pertaining to the primes that ramify in $K$. Among these invariants, the \emph{shape} of a number field is invariant with scaling and orthogonal transformations. Recently, this notion has been a natural testing ground for equidistribution phenomena in arithmetic statistics.

\par The Minkowski embedding $J: K \hookrightarrow K_{\mathbb{R}}:=K \otimes_{\Q} \mathbb{R}$
identifies $K$ with a real vector space equipped with an inner product induced by the trace form
\[
\langle x,y\rangle = \op{Tr}_{K/\Q}(x\bar y).
\]
The image $J(\cO_K)$ is a full-rank lattice in this space. The lattice necessarily contains the vector $J(1)$. Project $J(\cO_K)$ onto the hyperplane orthogonal to $J(1)$, or equivalently consider the lattice obtained from the trace-zero submodule
\[
\cO_K^\perp := \{ n\alpha - \op{tr}(\alpha) \mid \alpha \in \cO_K \}.
\]
The space of shapes of full-rank lattices in an $(n-1)$-dimensional real inner product space is given by the double coset space
\[
\mathcal{S}_{n-1}
=
\operatorname{GL}_{n-1}(\Z)\backslash
\operatorname{GL}_{n-1}(\mathbb{R})/
\operatorname{GO}_{n-1}(\mathbb{R}),
\]
where $\operatorname{GO}_{n-1}(\mathbb{R})$ denotes the group of linear transformations preserving the inner product up to a positive scalar multiple. Equivalently, identifying $\operatorname{GL}_{n-1}(\mathbb{R})/\operatorname{GO}_{n-1}(\mathbb{R})$ with the space $\mathcal{G}$ of positive definite symmetric $(n-1)\times (n-1)$ matrices modulo scaling, one obtains
\[
\mathcal{S}_{n-1}
=
\operatorname{GL}_{n-1}(\Z)\backslash \mathcal{G}/\mathbb{R}^\times.
\]
The \emph{shape of $K$} is defined to be the equivalence class of $J(\cO_K^\perp)$ in $\mathcal{S}_{n-1}$. There is a natural measure $\mu$ on $\mathcal{S}_{n-1}$, arising as the pushforward of Haar measure on $\operatorname{GL}_{n-1}(\mathbb{R})$. The space of shapes serves as the ambient space in which equidistribution questions for shapes are formulated. 

The systematic study of distribution questions for shapes of number fields was initiated by Terr \cite{Terr97}, who proved that the shapes of cubic fields are equidistributed in $\mathcal{S}_2$ with respect to $\mu$. In the cubic setting, shapes may be identified with points in the complex upper half-plane modulo $\operatorname{SL}_2(\Z)$, and Terr’s result asserts equidistribution in the classical fundamental domain endowed with the hyperbolic measure. Subsequent work by Bhargava and Shnidman \cite{BS14} classified shapes of cubic orders with automorphism group $\Z/3\Z$, while Mantilla-Soler and Monsurró \cite{MSM16} determined shapes of cyclic number fields of prime degree. Bhargava and Harron \cite{BH16} proved shape equidistribution results for $S_n$-number fields of degrees $n=4,5$, ordered by discriminant. Their work establishes that, for a fixed signature, the shapes of $S_n$-fields become equidistributed in $\mathcal{S}_{n-1}$ with respect to $\mu$. The proof relies crucially on the existence of parametrizations of cubic, quartic, and quintic orders due to Delone–Faddeev \cite{DeloneFaddeev} and Bhargava \cite{HighercompositionIII, HighercompositionIV}.

For degrees $n\ge 6$, no such parametrizations are known, and consequently the techniques underlying the equidistribution results for $n\le 5$ do not extend directly. As emphasized in \cite{BH16}, it remains an important open problem to understand the distribution of shapes in families of number fields of higher degree, especially for families whose Galois groups are non-generic. One particularly natural and tractable family in this direction is the family of \emph{pure number fields} of degree \(n\). These are fields of the form \(K=\Q(\sqrt[n]{m})\), where \(m\) is a nonzero \(n\)-th-power-free integer. Equivalently, this is the family of number fields \(K\) with \([K:\Q]=n\) whose Galois closure \(\widetilde{K}\) has Galois group isomorphic to the semidirect product \(\Z/n\Z \ltimes (\Z/n\Z)^\times\), and for which the associated resolvent field is the cyclotomic field \(\Q(\mu_n)\). The distribution of shapes in pure cubic fields was studied by Harron \cite{Har17}. For cubic fields, the shape is a point in the fundamental domain for the action of $\op{GL}_2(\Z)$ on the complex upper half plane. The measure induced from the Haar measure on $\op{GL}_2(\mathbb{R})$ coincides with the hyperbolic measure $\frac{dx\,dy}{y^2}$. In this case, the shapes lie on one of two loci. If $m\not \equiv \pm 1\pmod{9}$, then the shape lies on the imaginary axis. Otherwise, it lies on the line $\op{Re}(z)=1/3$. Harron shows that on each axis the shapes are equidistributed with respect to the hyperbolic measure. Holmes \cite{Hol22} treated pure number fields of prime degree $p>2$. In this case, the shapes lie on one of two affine spaces of dimension $(p-1)/2$ depending on whether $p$ is wildly or tamely ramified in $K$. Once again the shapes on each affine space is shown to be equidistributed with respect to the restriction of $\mu$. The situation becomes considerably more complicated when $n$ is composite. When $n=4$, it is shown in \cite{purequartic} that the shapes of pure quartic fields lie in five disjoint torus orbits, and is equidistributed by a measure on each orbit which is the product of a continuous measure and a discrete measure. Unlike the case when $n$ is a prime, this measure is not the restriction of $\mu$ to each torus orbit.
\subsection{Main results}

The present paper investigates the distribution of shapes in families of \emph{pure sextic fields} $K=\Q(\sqrt[6]{m})$,
ordered by absolute discriminant. Throughout, $m$ denotes a nonzero sixth-power-free integer. We write $m=\pm a_1 a_2^2 a_3^3 a_4^4 a_5^5$, with $a_1,\dots,a_5$ pairwise coprime positive integers. The discriminant of $K$ is then given by
\begin{equation}\label{v_1 v_2}
\Delta_K=\operatorname{sgn}(m)\, 2^{v_1}3^{v_2} a_1^5 a_2^4 a_3^3 a_4^4 a_5^5,
\end{equation}
where $v_1, v_2$ are as in \cite[Theorem 3.1]{J21}. This formula allows one to order pure sextic fields by discriminant and to parametrize them by integer tuples $(a_1,\dots,a_5)$ subject to the condition that $\prod_{i=1}^5 a_i$ is squarefree. An explicit integral basis for $\cO_K$ is computed in \cite{J21} and it depends on the residue class of $m$ modulo $2^63^5$. There are $20$ distinct cases which are naturally indexed by pairs $(i,j)$ with $1\leq i\leq 5$ and $1\leq j\leq 4$, as described in Theorem~\ref{Jakthm}. We refer to the pair $(i,j)$ as the \emph{Type} of $m$. Throughout this paper, we fix the sign of $m$ and fix a Type $(i,j)$, thereby restricting attention to a single arithmetic subfamily of pure sextic fields.

For a fixed sign $\epsilon\in\{+1,-1\}$ and a fixed Type $(i,j)$, let $\mathscr{R}\subset \mathbb{R}^3$ be a compact cuboidal region of the form
\[
\mathscr{R}=[R_1',R_1]\times [R_2',R_2]\times [R_3',R_3].
\]
Given a pure sextic field $K=\Q(\sqrt[6]{m})$ with
\[
m=\epsilon a_1 a_2^2 a_3^3 a_4^4 a_5^5,
\]
we associate to $K$ a triple of real parameters
\[
\lambda(m)=(\lambda_1(m),\,\lambda_2(m),\,\lambda_3(m)),
\]
defined by
\[
\lambda_1(m)=\left(\frac{a_4a_5^2}{a_1^2a_2}\right)^{1/3}, 
\qquad 
\lambda_2(m)=\left(\frac{a_2a_5}{a_1a_3^3a_4}\right)^{1/3}, 
\qquad 
\lambda_3(m) = \frac{1}{a_2a_4}.
\]
The shape of $K$ is represented by
\begin{equation}\label{shape eqn}C_{i,j}^{\,T}
\op{diag}\!\left(
\lambda_1(m),\,
\lambda_2(m),\,
\lambda_3(m),\,
\lambda_2(m)^{-1}/a_3^2,\,
\lambda_1(m)^{-1}
\right)
C_{i,j},
\end{equation}
where $C_{i,j}$ is an explicit matrix depending only on the Type $(i,j)$. 
\par Our main result establishes an asymptotic formula for the number of pure sextic fields of fixed sign and Type whose shape parameters $(\lambda_1(m),\, \lambda_2(m),\, \lambda_3(m))$ lie in a prescribed region.

\begin{lthm}[Theorem~\ref{main thm}]
With respect to the notation above, one has
\[
\lim_{N\rightarrow \infty}
\frac{
\#\left\{
K \;\middle|\;
m \text{ has sign } \epsilon \text{ and Type } (i,j),
(\lambda_1(m)^3,\, \lambda_2(m)^3,\, \lambda_3(m)^{-1})\in \mathscr{R},\,
|\Delta_K|\le 2^{v_1}3^{v_2}N
\right\}
}{N^{1/5}}
=
\int_{\mathscr{R}} \hat{\mu}_{i,j},
\]
where $\hat{\mu}_{i,j}$ is an explicit measure on $\mathbb{R}^3$ given by \eqref{mu i,j defn} which depends only on the Type $(i,j)$. Further, $\hat{\mu}_{i,j}$ is a product of continuous measures in the first and second coordinates and a discrete measure on the third coordinate.
\end{lthm}

\par
We set
\[
\lambda_4(m):=\frac{\lambda_2(m)^{-1}}{a_3^2},
\]
and it is therefore natural to study the distribution of the vector
\[
(\lambda_1(m),\, \lambda_2(m),\, \lambda_3(m),\, \lambda_4(m),\, \lambda_1(m)^{-1}).
\]
Note that
\[
a_3^2=(\lambda_2(m)\lambda_4(m))^{-1},
\]
and hence, for any bounded region in the $(\lambda_2,\lambda_4)$--plane, the parameter $a_3$ can assume only finitely many values. Furthermore, if $(\lambda_3(m)^{-1},a_3)$ is restricted to a bounded set, then there are only finitely many possibilities for the triple $(a_2,a_3,a_4)$. Fix such a triple and write $a':=(a_2,a_3,a_4)$. For each fixed $a'$, the parameters $\lambda_3(m)$ and $\lambda_4(m)$ are completely determined, while the remaining parameters $(\lambda_1(m),\lambda_2(m))$ vary along a one-parameter curve
\[
\mathcal{C}_{a'}:=\{(c_1 t^2,\, c_2 t^{-1}) \mid t>0\},
\]
where the constants $c_1$ and $c_2$ depend only on $a'$, and the parameter $t:=\left(\frac{a_5}{a_1}\right)^{1/3}$. In particular, the vector $(\lambda_1(m),\dots,\lambda_4(m), \lambda_1(m)^{-1})$ lies on one of infinitely many curves indexed by the finitely supported arithmetic data $a'$, and its position along such a curve is determined by the ratio $a_5/a_1$. In light of this, the distribution of the shape is determined by the triple $(a_5/a_1, a_3, a_2a_4)$ where the first parameter is continuously varying and the second and third parameters take on discrete values. 

\begin{lthm}[Theorem~\ref{main thm 2}]
With respect to the notation above, one has
\[
\lim_{N\rightarrow \infty}
\frac{
\#\left\{
K \;\middle|\;
m \text{ has sign } \epsilon \text{ and Type } (i,j),
(a_5/a_1,\, a_3,\, a_2a_4)\in \mathscr{R},\,
|\Delta_K|\le 2^{v_1}3^{v_2}N
\right\}
}{N^{1/5}}
=
\int_{\mathscr{R}} \hat{\nu}_{i,j},
\]
where $\hat{\nu}_{i,j}$ is an explicit measure given by \eqref{nu i, j defn} as a product of a continuous measure in the first coordinate and discrete measures in the second and third coordinates, and depends only on the Type $(i,j)$.
\end{lthm}

\subsection{Methodology} The proofs of these results combine methods from the geometry of numbers with analytic sieve techniques. The problem of counting pure sextic fields with prescribed shape parameters lying in a fixed cuboidal region is first translated into a question about counting lattice points in explicitly described regions of Euclidean space. By an application of Davenport’s lemma, the number of lattice points in such regions is shown to be well approximated by their volumes, up to acceptable error terms. However, not all lattice points correspond to sextic fields: those arising from pure sextic fields are characterized by an infinite collection of congruence conditions reflecting local arithmetic constraints. To isolate the contribution of these admissible points and obtain sharp asymptotics, we therefore employ sieve methods. A further complication arises from the fact that the relevant lattice points do not occupy a single full-dimensional region, but rather lie on infinitely many lower-dimensional regions that foliate the ambient space into parallel leaves. On each such leaf, the admissible lattice points are counted separately, taking into account both the geometric constraints and the arithmetic congruence conditions. The final asymptotic formula is obtained by summing the contributions from all leaves, using a careful analysis to control convergence and error terms.
\subsection{Outlook}
Taken together, these results provide a detailed description of how shapes are distributed within the family of pure sextic fields, once one fixes the natural arithmetic refinements given by the sign of \(m\) and the local Type \((i,j)\). The appearance of limiting measures that factor as products of continuous and discrete components highlights a new phenomenon in the study of shapes. Looking ahead, the methods developed here suggest several avenues for further investigation. A natural next step is to extend this analysis to pure number fields of higher composite degree, for instance when the degree is the square of a prime number. One may also consider other sparse families of number fields of fixed degree for which explicit integral bases have been calculated.

\subsection{Organization} Including the introduction, the article consists of five sections. In Section~\ref{s 2}, we review the notion of the shape of a number field, recalling its definition in terms of the Minkowski embedding. We also summarize the basic properties of shapes that are relevant for the present work. We then turn to pure number fields \(K=\Q(\sqrt[n]{m})\), describing general results on the structure of their rings of integers, explicit integral bases, and discriminants. The general discussion gives one a good idea how the results in this paper can be generalized to other values of $n$. For pure sextic fields, the integral bases are determined by local conditions at the primes $2$ and $3$. These decompose into \(20\) distinct Types \((i,j)\) which will play a central role in the subsequent analysis.
For each Type, Theorem \ref{Jakthm} recalls an explicit integral basis of $\cO_K$. Section~\ref{s 3} is devoted to the explicit computation of shapes in each of these $20$ cases. We determine for every Type $(i,j)$, an explicit matrix $C_{i,j}$ such that the shape of a pure sextic field $K=\Q(\sqrt[6]{m})$ is given by the expression \eqref{shape eqn}. Finally in Sections \ref{s 4} and \ref{s 5}, we prove the main results of the article concerning the asymptotic distribution of $(\lambda_1,\, \lambda_2,\, \lambda_3)$ and $(a_5/a_1,\, a_3,\, a_2a_4)$ as $m$ varies over sixth power free integers of fixed sign and fixed Type $(i,j)$. 

\subsection*{Acknowledgements} This paper was initiated as part of the project group  "shapes of number fields", initiated during the \href{https://sites.google.com/view/imscnumbertheorygroupmeeting/home}{Number theory working group meeting} held at the Institute of Mathematical Sciences from 27th to 31st March, 2025. The third author gratefully acknowledges the opportunity to lead this project and thanks the institute for providing a stimulating and collaborative research environment. 

\subsection*{Data Availability} There is no data associated to the results of this manuscript.

\section{Preliminaries}\label{s 2}
\subsection{Notation}

We summarize here the notation used throughout this article.

\begin{itemize}
\item Given a complex number $z$, we denote its real and imaginary parts by $\Re z$ and $\Im z$, respectively.
\item For a real number $\lambda$, we denote by $\lfloor \lambda \rfloor$ the greatest integer not exceeding $\lambda$. 
\item Given a prime number $\ell$, let $v_\ell(\cdot)$ be the valuation normalized by $v_\ell(\ell)=1$. Denote by $|\cdot|_\ell$ the absolute value normalized by $|\ell|_\ell=\ell^{-1}$.
\item Let $K/\mathbb{Q}$ be a number field of degree $n$, and let $\mathcal{O}_K$ denote its ring of integers. Let $\Delta_K$ be the discriminant of $K$ over $\Q$. The field $K$ admits $r$ real embeddings $\sigma_1, \ldots, \sigma_r$ and $s$ pairs of complex conjugate embeddings $\tau_1, \ldots, \tau_s$, so that $n = r + 2s$. These embeddings define the \emph{Minkowski embedding}:
\[
J : K \hookrightarrow \mathbb{C}^n, \quad \alpha \mapsto \left( \sigma_1(\alpha), \ldots, \sigma_r(\alpha), \tau_1(\alpha), \overline{\tau_1(\alpha)}, \ldots, \tau_s(\alpha), \overline{\tau_s(\alpha)} \right).
\]
Let $K_{\mathbb{R}}$ denote the $\mathbb{R}$-linear span of $J(K)$, which we may identify with $\mathbb{R}^{r+2s}=\mathbb{R}^n$ via:
\[\alpha\mapsto \left( \sigma_1(\alpha), \ldots, \sigma_r(\alpha), \Re\tau_1(\alpha), \Im\tau_1(\alpha), \ldots, \Re\tau_s(\alpha), \Im\tau_s(\alpha)\right).\]

\item For a positive integer $l$, let $M_l(\mathbb{R})$ denote the space of $l \times l$ real matrices, and let $\operatorname{GL}_l(\mathbb{R})$ denote the group of invertible elements in $M_l(\mathbb{R})$.

\item Given $M \in M_l(\mathbb{R})$, we write $M^T$ for its transpose and $\operatorname{Id}_l$ for the identity matrix of size $l$. The \emph{general orthogonal group} $\operatorname{GO}_l(\mathbb{R})$ is the subgroup of $\operatorname{GL}_l(\mathbb{R})$ consisting of matrices $M$ satisfying
\[
M M^T = \lambda \operatorname{Id}_l
\]
for some scalar $\lambda \in \mathbb{R}^\times$.

\end{itemize}

\subsection{Shapes of number fields}

\par Given a number field $K$ of degree $n$. Note that $K_{\mathbb{R}}$ comes equipped with the standard inner product restricting to the trace form $\langle x,y \rangle =\op{Tr}_{K/\Q}(x\bar{y})$ on $K$. Let $V$ be the subspace of $\mathbb{R}^n$ which is orthogonal to $J(1)=(1, \dots, 1)$. We choose a basis $\{e_1, \dots, e_{n-1}\}$ of $V$ so that $\langle e_i, e_j\rangle=\delta_{i,j}$. Given a lattice $\Lambda\subset V$ of rank $(n-1)$, one is often interested not in the precise geometry of $\Lambda$, but in its \emph{shape}, i.e., in its equivalence class under isometries and scaling. Two lattices $\Lambda_1, \Lambda_2 \subset V$ are said to have the same \emph{shape} if there exists an orthogonal transformation $Q \in \operatorname{GO}_{n-1}(\mathbb{R})$ such that $\Lambda_2 =  Q \Lambda_1$. More formally, let $\{b_1, \dots, b_{n-1}\}$ be a $\Z$-basis of $\Lambda$, and write each $b_i$ in terms of the standard basis as
$
b_i = \sum_{j=1}^{n-1} a_{i,j} e_j
$. Then the matrix $A = (a_{i,j}) \in \operatorname{GL}_{n-1}(\mathbb{R})$ represents the change of basis from the standard basis to one for the lattice $\Lambda$. Since the choice of basis for $\Lambda$ is not canonical, and since the notion of shape is invariant under change of basis, scalar multiplication, and orthogonal transformation, the shape of $\Lambda$ is captured by the double coset of $A$ in the quotient

$$
\mathcal{S}_{n-1} := \operatorname{GL}_{n-1}(\mathbb{Z}) \backslash \operatorname{GL}_{n-1}(\mathbb{R}) / \operatorname{GO}_{n-1}(\mathbb{R}),
$$
\noindent where $\operatorname{GL}_{n-1}(\mathbb{Z})$ accounts for the change of basis in $\Lambda$, and $\operatorname{GO}_{n-1}(\mathbb{R})$ accounts for orthogonal transformations and scaling. The space $\mathcal{S}_{n-1}$ thus parametrizes the set of shapes of full rank lattices in $V$.

The study of shapes of number fields is motivated in part by the observation that certain arithmetic properties of $K$ are reflected in the geometry of the associated lattice $\Lambda$, and vice versa. In particular, the distribution of shapes of number fields is a rich and active area of investigation. To study such statistical questions, one equips $\mathcal{S}_{n-1}$ with a natural measure $\mu$, which arises as the pushforward of the Haar measure on $\operatorname{GL}_{n-1}(\mathbb{R})$ under the projection to the double coset space. This measure allows for the formulation of equidistribution results and density theorems for families of number fields via the geometry of their shapes.
\par An alternative and sometimes more convenient description of $ \mathcal{S}_{n-1} $ can be given in terms of positive definite symmetric matrices. Let $ \mathcal{G} $ denote the space of $ (n-1) \times (n-1) $ positive definite symmetric real matrices. The group $ \operatorname{GL}_{n-1}(\mathbb{Z}) $ acts on $ \mathcal{G} $ via $M \cdot G := M^TGM$, where $M \in \operatorname{GL}_{n-1}(\mathbb{Z})$ and $G \in \mathcal{G}$. Additionally, $ \mathbb{R}^\times $ acts on $ \mathcal{G} $ by scalar multiplication:  
$G \cdot r := r^2 G$, where $r \in \mathbb{R}^\times$. There is a natural $ \operatorname{GL}_{n-1}(\mathbb{Z}) $-equivariant bijection \[\operatorname{GL}_{n-1}(\mathbb{R}) / \operatorname{GO}_{n-1}(\mathbb{R}) \to \mathcal{G} / \mathbb{R}^\times\]
given by mapping a matrix $ M $ to $ M^TM $, thus providing a characterization of the space of lattice shapes as 
\begin{equation}\label{new interpretation}\mathcal{S}_{n-1} :=\operatorname{GL}_{n-1}(\mathbb{Z}) \backslash \mathcal{G}/\mathbb{R}^\times.\end{equation} Let $\op{Gr}(\Lambda):=(\langle b_i, b_j\rangle)_{i,j}$ be the Gram matrix of $\Lambda$; the shape of $\Lambda$ is the equivalence class of $\op{Gr}(\Lambda)$ in $\mathcal{S}_{n-1}$. Let $M\in \op{GL}_{n-1}(\mathbb{Z})$ be a change of basis matrix and $r\in \mathbb{R}^\times$, then
\[(\langle r M b_i, r M b_j\rangle)_{i,j}=r^2M^T(\langle b_i, b_j\rangle)_{i,j}M,\] and thus the shape of $\Lambda$ is well defined.
\par Let $ \{\sigma_i\}_{1 \leq i \leq n} $ denote the set of embeddings of $ K $ into $ \mathbb{C} $. Assume that $\sigma_1, \dots, \sigma_r$ are the real embeddings and $\tau_1, \overline{\tau_1} \dots, \tau_s, \overline{\tau_s}$ are pairs of complex conjugate embeddings. The Minkowski embedding  
\[
\begin{aligned}
    J: K &\to \mathbb{C}^{n} \\
    \alpha &\mapsto \left( \sigma_1(\alpha), \ldots, \sigma_r(\alpha), \tau_1(\alpha), \overline{\tau_1(\alpha)}, \ldots, \tau_s(\alpha), \overline{\tau_s(\alpha)} \right)
\end{aligned}
\]
identifies $ K $ with an $ n $-dimensional subspace of $ \mathbb{C}^n $. The real span of the image, denoted $ K_{\mathbb{R}} $, inherits a natural inner-product structure, and the restriction of $ J $ to $ \mathcal{O}_{K} $ yields a rank-$ n $ lattice in $ K_{\mathbb{R}} $. Although one might be inclined to define the shape of $ K $ as the shape of this lattice, such a definition presents difficulties when studying the distribution of shapes across a family of number fields: since each lattice necessarily contains the vector $ J(1) $, this introduces a nontrivial constraint, reducing randomness in the distribution of shapes. To circumvent this issue, one defines the shape of $ K $ in terms of a sublattice obtained by projecting $ J(\mathcal{O}_K) $ onto the orthogonal complement of $ J(1) $. 
\begin{definition}
Consider the map  
\[
\alpha^{\perp} := n\alpha - \operatorname{tr}(\alpha),
\]
which sends elements of $ \mathcal{O}_K $ to those of trace zero. The set $ \mathcal{O}_K^{\perp} $, consisting of the images of $ \mathcal{O}_K $ under this map, forms a rank-$(n-1)$ submodule of $ \mathcal{O}_K $, and its Minkowski embedding yields an associated rank-$(n-1)$ lattice $J(\cO_K^\perp)$. We define the shape of $ K $ to be the shape of this projected lattice.
\end{definition}Given a basis $ \{1, \alpha_1, \dots, \alpha_{n-1}\} $ of $ \mathcal{O}_K $, the corresponding basis for $ \mathcal{O}_K^{\perp} $ is given by $\beta= \{\alpha_1^{\perp}, \dots, \alpha_{n-1}^{\perp} \} $. The Gram matrix with respect to this basis is given by:  
\[
\operatorname{Gr}\left(J\left(\mathcal{O}_{K}^{\perp}\right)\right) = \big( \langle J(\alpha_i^{\perp}), J(\alpha_j^{\perp}) \rangle \big)_{1\leq i,j\leq n-1},
\]
and its class in $\mathcal{S}_{n-1}$ represents the shape of $K$.

The case $n=3$ admits a particularly transparent description. A general element of $\mathcal{G}/\mathbb{R}^\times$ is represented by a matrix $\begin{pmatrix} a & b \\ b & c \end{pmatrix}$ where $a>0$ and $ac-b^2=1$. One obtains an explicit $\op{GL}_2(\mathbb{R})$-equivariant identification of $\mathcal{G}/\mathbb{R}^\times$ with $\mathbb{H}$, the complex upper half-plane. Explicitly, the matrix above corresponds to the point $z=x+iy\in\mathbb{H}$ with
\[
x=\frac{b}{c}, \qquad y=\sqrt{\frac{a}{c}-x^2}.
\]
Under this correspondence, the action of $\op{GL}_2(\mathbb{Z})$ becomes the usual fractional linear action on $\mathbb{H}$, and a fundamental domain is given by
\[
\mathcal{F}_2=\{\, x+iy\in\mathbb{H} \mid x\in[0,1/2),\; y>0,\; x^2+y^2\ge 1 \,\},
\]
as depicted in Figure 1 below.
\begin{figure}[h]
\centering
\begin{tikzpicture}[scale=3]

\draw[->] (-1.2,0) -- (1.2,0) node[right] {$x$};

\draw[thick] (0.5,0) -- (0.5,2.0);

\draw[thick] (1,0) arc (0:180:1);

\fill[gray!25]
  (0,{sqrt(1-0.5*0.5)})
  --
  (0,2.0)
  --
  (0.5,2.0)
  --
  (0.5,{sqrt(1-0.5*0.5)})
  arc (60:90:1)
  -- cycle;
\draw[->] (0,0) -- (0,2.0) node[above] {$y$};
\node at (0.25,1.4) {$\mathcal{F}_2$};
\node[below] at (0,0) {$0$};
\node[below] at (0.5,0) {$\tfrac12$};

\end{tikzpicture}
\caption{The fundamental domain $\mathcal{F}_2 \subset \mathbb{H}$ for shapes of rank-$2$ lattices.}
\end{figure}
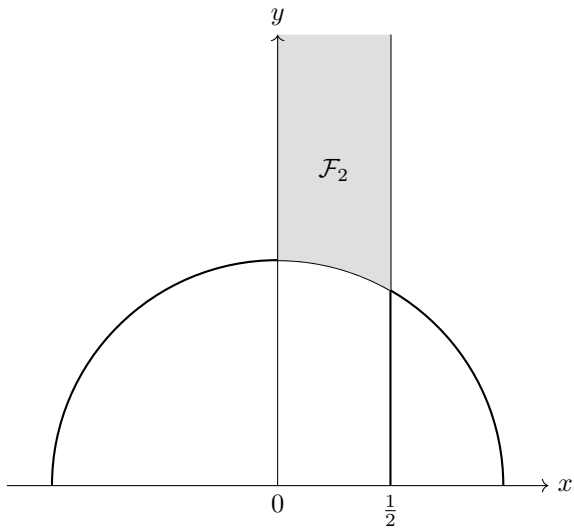The induced measure on $\mathcal{F}_2$ is the hyperbolic measure $\frac{dx\, dy}{y^2}$ coming from Haar measure on $\op{GL}_2(\mathbb{R})$.
On the other hand for $n>3$, the geometry of the space of lattice shapes becomes substantially more
complicated. In this context, a fundamental domain was computed by Minkowski \cite{Minkowski} (see also \cite{Grenier} for further details).
\subsection{Integral bases of pure number fields}
In this section, we shall discuss generalities pertaining to integral bases for \emph{pure number fields}, i.e., fields of the form $ K = \mathbb{Q}(\theta) $ where $\theta$ is a root of an irreducible polynomial $x^n-m$ where $n>1$ is an integer and $m$ is an $n$-th power-free integer. The following assumption is in place only to simplify our discussions and will only be assumed in this section.
\begin{ass}\label{basic ass section 2.3}
    Assume that $m$ is an $ n $-th power free and that for each prime number $\ell|n$, either $ v_{\ell}(m) = 0 $ or $ v_{\ell}(m) $ is coprime to $ \ell $.
\end{ass}
Let the prime factorizations of $ n $ and $ |m| $ be given by $n= \prod_{i=1}^{k} p_i^{s_i}$ and $|m| = \prod_{j=1}^{l} q_j^{t_j}$. Here, $p_i$ and $q_j$ are prime numbers. Thus by assumption, if $q_j|n$, then $t_j$ is coprime to $q_j$.
Define $ m_j = \gcd(n, t_j) $, and set $ n_i = \frac{n}{p_{i}^{s_i}} $ and  
\[
r_i = v_{p_i}(m^{p_i - 1} - 1) - 1.
\]
\noindent Then the discriminant of $ K $ is given by  
\[
\Delta_K = (-1)^{\frac{(n-1)(n-2)}{2}} \operatorname{sgn}(m^{n-1}) \left( \prod_{i=1}^{k} p_i^{v_i} \right) \prod_{j=1}^{l} q_j^{n - m_j},
\]  
where  
\[
v_i =
\begin{cases}
n s_i - 2 n_i \sum_{j=1}^{\min \{ r_i, s_i \} } p_i^{s_i - j}, & \text{if } r_i > 0, \\
n s_i, & \text{otherwise,}
\end{cases}
\]
see \cite[Theorem 1.A]{JKSMathematika}. 

\par Before describing an integral basis for $\cO_K$, we introduce some further notation. The conventions used here are consistent with those in \emph{op. cit.} In what follows, let $\theta$ be a root of the irreducible polynomial $x^n - m \in \mathbb{Z}[x]$, where $|m|$ is $ n $-th power-free and can be expressed as $ \prod_{j=1}^{n-1} a_j^j $, with $ a_1, a_2, \dots, a_{n-1} $ being squarefree positive integers that are pairwise relatively prime. Such a tuple $(a_1, \dots, a_{n-1})$ is called a \emph{strongly carefree couple}. For $0 \leqslant i \leqslant n-1$, define  
\[
C_i = \prod_{j=1}^{n-1} a_j^{\lfloor ij / n \rfloor}.
\]  
Since $\theta^n = m$, it follows that the element $ \frac{\theta^i}{C_i} $ satisfies the monic polynomial  
\[
x^n - \operatorname{sgn}(m)^i \prod_{j=1}^{n-1} a_j^{ij - n \lfloor ij  / n \rfloor}
\]  
and is therefore an algebraic integer.
\par Let us define 
\begin{equation}\label{defn of S}
S:= \left\{ i \in \{1, \dotsc, k\} \mid r_i > 0 \right\}    
\end{equation}
as the set of indices for which $ r_i $ is positive. For each $ i \in S $, set $
d_i = \min\{r_i, s_i\}$. Note that $ p_i \nmid m $ for all $ i \in S $, since otherwise we would have $ v_{p_i}(m^{p_i - 1} - 1) = 0 $, implying $ r_i = -1 < 0 $, contradicting the assumption that $ i \in S $. For each such index $ i $, the interval $ [0, n - 1] $ admits a partition into pairwise disjoint subintervals of the form  
\[
\bigcup_{k_i = 0}^{d_i - 1} \left[n - \frac{n}{p_i^{k_i}},\ n - \frac{n}{p_i^{k_i + 1}} \right) \cup \left[n - \frac{n}{p_i^{d_i}},\ n\right),
\]  
which together cover all integers $ t \in [0, n - 1] $. Thus, for each such $ t$ and $ i \in S $, there exists a unique integer $ k_{i,t} $ with $ 0 \leq k_{i,t} \leq d_i $, and a corresponding integer $ j_{i,t} \geq 0 $, such that  
\[
t = n - \frac{n}{p_i^{k_{i,t}}} + j_{i,t}.
\]  
\noindent Now fix an integer $ t\in [0, n - 1] $, and define  
\[
S_t := \left\{ i \in S \mid k_{i,t} \geq 1 \right\}.
\]  
Choose any $ i \in S_t $, and fix such a pair $ (i, t) $. Since $ p_i \nmid m $, one can find an integer $ b'_{i,t} $ satisfying  
\[
mb'_{i,t} \equiv 1 \pmod{p_i^{k_{i,t} + 1}}.
\]  
Define  
\[
a'_{i,t} = 
\begin{cases}
(b'_{i,t})^{p_i^{s_i - 1 - k_{i,t}}} & \text{if } 1 \leq k_{i,t} < s_i, \\
b'_{i,t} & \text{if } k_{i,t} = s_i.
\end{cases}
\]  
\noindent Let us denote  
\[
n_{i,t} = \frac{n}{p_i^{k_{i,t}}}.
\]  
We now choose an integer $ w_{i,t} \in \mathbb{Z} $ such that  
\[
w_{i,t} C_t (a'_{i,t})^{p_i^{k_{i,t}} - 1} \equiv 1 \pmod{p_i^{k_{i,t}}},
\]  
which is possible since $ p_i \nmid C_t a'_{i,t} $. Define the element $ \delta_{i,t} \in \mathbb{Z}[\theta] $ by  
\[
\delta_{i,t} = w_{i,t} C_t \theta^{j_{i,t}} \sum_{r = 0}^{p_i^{k_{i,t}} - 2} \left(a'_{i,t} \theta^{n_{i,t}}\right)^r.
\]  
\noindent Now set  
\[
z_{i,t} = \prod_{j \in S_t \setminus \{i\}} p_j^{k_{j,t}},
\]  
for each $ i \in S_t $. Since the integers $ z_{i,t} $ are pairwise coprime and satisfy $ \gcd\{z_{i,t} \mid i \in S_t\} = 1 $, there exist integers $ u_{i,t} \in \mathbb{Z} $ such that  
\[
\sum_{i \in S_t} u_{i,t} z_{i,t} = 1.
\]  
We then define  
\[
\beta_t =
\begin{cases}
\sum\limits_{i \in S_t} u_{i,t} z_{i,t} \delta_{i,t}, & \text{if } S_t \neq \emptyset, \\
0, & \text{if } S_t = \emptyset.
\end{cases}
\]  
\noindent It remains to note that the largest power of $ \theta $ appearing in $ \delta_{i,t} $ is  
\[
j_{i,t} + (p_i^{k_{i,t}} - 2) n_{i,t} = j_{i,t} + n - 2n_{i,t} = t - n_{i,t} < t,
\]  
and hence all powers of $ \theta $ occurring in $ \delta_{i,t} $ are strictly less than $ t $. Finally, in the special case where $ S_t = \{i\} $ is a singleton, the product $ z_{i,t} $, being taken over an empty set, is equal to 1, so $ \beta_t = \delta_{i,t} $ in this case. With the above notation in place, the following result describes an integral basis of pure number fields. 
\begin{theorem}\label{3.1} Let $K = \Q(\theta)$ with $\theta$ having minimal polynomial $x^n - m$ over $\Q$ where 
$m$ is an $n$-th power free integer and for every prime $p_i$ dividing $n$ either $p_i \nmid m$ or $v_{p_i}(m)$ is 
coprime to $p_i$. Let $n = \prod\limits_{i=1}^{k}p_i^{s_i}, S, C_m, k_{i,m}$ and $\beta_m$ be as in the 
above paragraph. Then the following assertions hold:\\
  $(i)$ If $S=\emptyset$, then $ \left\lbrace  \dfrac{\theta^t}{C_t } ~\bigg\vert~ 0\leq t\leq n-1  \right\rbrace $ is an integral basis of $K$.   \\
  $(ii)$ If $S\neq \emptyset$, then $ \left\lbrace 1,  \dfrac{\theta^t +\beta_t}{C_t \prod\limits_{i\in S} p_i ^{k_{i,t}}} ~\bigg\vert~ 1\leq t\leq n-1  \right\rbrace $ is an integral basis of $K$. 
\end{theorem}   

\begin{proof}
    See \cite[Theorem 1.6]{JKSMathematika}.
\end{proof}

\begin{remark}\label{3.5'} In the paragraph preceding Theorem \ref{3.1}, note that $b'_{i,t}$ can be chosen to be any integer which is congruent to $m^{p_i-2}$ modulo  $p_i^{k_{i,t}+1}$ because $m^{p_i -1} \equiv 1~(mod~p_i^{r_i+1})$ and $k_{i,t} \leq r_i$.
\end{remark}
\subsection{Integral bases of pure sextic fields}
\par We now describe the integral bases for pure sextic fields. Let $K = \Q(\theta)$ with $\theta$  a real root of an irreducible polynomial $f(x) = x^6 - m$, where $m$ is a sixth power-free integer. Write $|m|= a_1a_2^2a_3^3a_4^4a_5^5$, where $a_1, a_2, a_3, a_4, a_5$ are square-free positive integers for which $\op{gcd}(a_i, a_j)=1$ for $1\leq i<j\leq 5$. Let \begin{equation}\label{C_k defn}C_{i}:= \prod\limits_{j=1}^{5}a_j^{\left\lfloor\frac{ij}{6}\right\rfloor}\end{equation} for $1\leq i \leq 5$. In contrast to Theorem \ref{3.1}, no assumptions on the parameter $m$ are required. In what follows, we shall set $\mathfrak{B}_{i, j}$ to be an integral basis obtained from combining the congruence conditions $(Ai)$ and $(Bj)$.

\begin{theorem} \cite[Theorem 1.1]{J21} \label{Jakthm}With respect to notation above,
$\mathfrak{B}_{i,j}=\{1,~ \theta,~ \dfrac{\theta^2}{C_2},~\beta, ~\gamma, ~\delta\}$ is an integral basis of $K$, where the algebraic integers $\beta, \gamma$ and $\delta$ are given in Table 2, according to the value of $m$ as given in Table 1. In Table 2, the numbers $m_3, C_{33}, C_{43}$, $C_{53}$ and $C_{52}$ stand respectively for $\frac{m}{27}, \frac{C_3}{3}, \frac{C_4}{9}$, $\frac{C_5}{9}$ and $\frac{C_5}{8}$. 
\end{theorem}

\begin{minipage}{.48\textwidth}   
\begin{center}
Cases based on the prime $2$
\begin{longtable}[h!]{|m{0.7cm}|m{2.75cm}|} 
\hline 
\multirow{4}{*}{A1} & $m \equiv 2$ mod 4\\ 
& $m \equiv 3$ mod 4\\
& $m \equiv 8$ mod 16\\
& $m \equiv 32$ mod 64\\ \hline
\multirow{2}{*}{A2} & $m \equiv 1$ mod 4\\ 
 & $m \equiv 4$ mod 16\\ \hline
A3 & $m \equiv 12$ mod 16\\ \hline
A4 & $m \equiv 16$ mod 64\\ \hline
A5 & $m\equiv 48$ mod 64 \\ \hline
\end{longtable}
\end{center}
    \end{minipage}  \hfill  
    \begin{minipage}{0.48\textwidth}
 \begin{center}
 Cases based on the prime $3$
\begin{longtable}[h!]{|m{0.7cm}|m{5.75cm}|} 
\hline 
\multirow{4}{*}{B1} & $m \equiv 2,3,4,5,6$ or $7$ mod 9\\ 
& $m \equiv 9$ or $18$ mod 27\\
& $m \equiv 81$ or $162$ mod 243\\
& $m \equiv 243$ or $486$ mod 729\\ \hline
{B2} & $m \equiv 1$ or $8$ mod 9\\ \hline
B3 & $m \equiv 27$ or $216$ mod 243\\ \hline
B4 & $m \equiv 54, 108, 135$ or $189$ mod 243\\ \hline
\end{longtable}
\end{center}     
    \end{minipage}
    \begin{center}
    Table 1: Cases according to the value of $m\pmod{2^63^5}$.
\end{center}

\begin{center}
\begin{longtable}[h!]{ |m{0.7cm}|m{0.7cm}|m{3.5cm}|m{4.7cm}|m{4.7cm}|} 
\hline
 \multicolumn{2}{|c|}{ Case } &  $\beta $ & $\gamma$ & $\delta$ \rule{0pt}{12pt} \\ \hline
 \multirow{4}{*}{A1}& B1 & $\frac{\theta^3}{C_3}$  & $\frac{\theta^4}{C_4}$ & $\frac{\theta^5}{C_5}$ \bigstrut \rule{0pt}{18pt} \\ \cline{2-5}
& B2 &$\frac{\theta^3}{C_3}$  & $\frac{\theta^4+mC_4^2\theta^2+C_4^2}{3C_4}$ & $\frac{\theta^5+mC_5^2\theta^3+C_5^2\theta}{3C_5}$ \bigstrut \rule{0pt}{18pt} \\ \cline{2-5}
& B3 & $\frac{\theta^3+6m_3C_{33}^2\theta}{3C_3}$  & $\frac{\theta^4+3m_3C_{43}^2\theta^2 + 9C_{43}^2}{3C_4}$ & $\frac{\theta^5+3m_3C_{53}^2\theta^3 + 9C_{53}^2\theta}{3C_5}$ \bigstrut \rule{0pt}{18pt} \\ \cline{2-5}
& B4 &  $\frac{\theta^3}{C_3}$  & $\frac{\theta^4}{C_4}$ & $\frac{\theta^5+3m_3C_{53}^2\theta^3 + 9C_{53}^2\theta}{3C_5}$ \bigstrut \rule{0pt}{18pt} \\ \cline{1-5}

\multirow{4}{*}{A2}& B1 & $\frac{\theta^3+C_3}{2C_3}$  &$\frac{\theta^4+C_4\theta}{2C_4}$ & $\frac{\theta^5+C_5\theta^2}{2C_5}$ \bigstrut \rule{0pt}{18pt} \\ \cline{2-5}
& B2 & $\frac{\theta^3+C_3}{2C_3}$  & $\frac{\theta^4-2mC_4^2\theta^2+3C_4\theta-2C_4^2}{6C_4}$ &$\frac{\theta^5-2mC_5^2\theta^3+3C_5\theta^2-2C_5^2\theta}{6C_5}$ \bigstrut \rule{0pt}{18pt} \\ \cline{2-5}
& B3 & $\frac{\theta^3-12m_3C_{33}^2\theta+3C_3}{6C_3}$  &$\frac{\theta^4-6m_3C_{43}^2\theta^2+3C_4\theta-18C_{43}^2}{6C_4}$ & $\frac{\theta^5-6m_3C_{53}^2\theta^3+3C_5\theta^2-18C_{53}^2\theta}{6C_5}$ \bigstrut \rule{0pt}{18pt} \\ \cline{2-5}
& B4 & $\frac{\theta^3+C_3}{2C_3}$  &$\frac{\theta^4+C_4\theta}{2C_4}$ &$\frac{\theta^5-6m_3C_{53}^2\theta^3+3C_5\theta^2-18C_{53}^2\theta}{6C_5}$ \bigstrut \rule{0pt}{18pt} \\ \cline{1-5}

\multirow{4}{*}{A3}& B1 &  $\frac{\theta^3}{C_3}$  &$\frac{\theta^4}{C_4}$ & $\frac{\theta^5+C_5\theta^2}{2C_5}$ \bigstrut \rule{0pt}{18pt} \\ \cline{2-5}
& B2 & $\frac{\theta^3}{C_3}$  & $\frac{\theta^4+mC_4^2\theta^2+C_4^2}{3C_4}$ & $\frac{\theta^5-2mC_5^2\theta^3+3C_5\theta^2-2C_5^2\theta}{6C_5}$\bigstrut \rule{0pt}{18pt} \\ \cline{2-5}
& B3 & $\frac{\theta^3+6m_3C_{33}^2\theta}{3C_3}$  & $\frac{\theta^4+3m_3C_{43}^2\theta^2+9C_{43}^2}{3C_4}$ & $\frac{\theta^5-6m_3C_{53}^2\theta^3+3C_5\theta^2-18C_{53}^2\theta}{6C_5}$ \bigstrut \rule{0pt}{18pt} \\ \cline{2-5}
& B4 &$\frac{\theta^3}{C_3}$  &$\frac{\theta^4}{C_4}$ &$\frac{\theta^5-6m_3C_{53}^2\theta^3+3C_5\theta^2-18C_{53}^2\theta}{6C_5}$ 
 \bigstrut \rule{0pt}{18pt} \\ \cline{1-5}
\hline
 \multirow{4}{*}{A4}& B1 & $\frac{\theta^3+C_3}{2C_3}$  &$\frac{\theta^4+C_4\theta}{2C_4}$ &$\frac{\theta^5+4C_{52}\theta^2}{2C_5}$ \bigstrut \rule{0pt}{18pt} \\ \cline{2-5}
& B2 & $\frac{\theta^3+C_3}{2C_3}$  &$\frac{\theta^4-2mC_4^2\theta^2+3C_4\theta-2C_4^2}{6C_4}$ &$\frac{\theta^5-2mC_5^2\theta^3+12C_{52}\theta^2-2C_5^2\theta}{6C_5}$ \bigstrut \rule{0pt}{18pt} \\ \cline{2-5}
& B3 &$\frac{\theta^3-12m_3C_{33}^2\theta+3C_3}{6C_3}$  &$\frac{\theta^4-6m_3C_{43}^2\theta^2+3C_4\theta-18C_{43}^2}{6C_4}$ & $\frac{\theta^5-6m_3C_{53}^2\theta^3+12C_{52}\theta^2-18C_{53}^2\theta}{6C_5}$\bigstrut \rule{0pt}{18pt} \\ \cline{2-5}
& B4 &$\frac{\theta^3+C_3}{2C_3}$  &$\frac{\theta^4+C_4\theta}{2C_4}$ &$\frac{\theta^5-6m_3C_{53}^2\theta^3+12C_{52}\theta^2-18C_{53}^2\theta}{6C_5}$ \bigstrut \rule{0pt}{18pt} \\ \cline{1-5}
 \multirow{4}{*}{A5}& B1 &$\frac{\theta^3}{C_3}$  &$\frac{\theta^4+C_4\theta}{2C_4}$ & $\frac{\theta^5}{C_5}$ \bigstrut \rule{0pt}{18pt}\\ \cline{2-5}
& B2 & $\frac{\theta^3}{C_3}$  &$\frac{\theta^4-2mC_4^2\theta^2+3C_4\theta-2C_4^2}{6C_4}$  & $\frac{\theta^5+mC_5^2\theta^3+C_5^2\theta}{3C_5}$ \bigstrut \rule{0pt}{18pt} \\ \cline{2-5}
& B3 & $\frac{\theta^3+6m_3C_{33}^2\theta}{3C_3}$ &$\frac{\theta^4-6m_3C_{43}^2\theta^2+3C_4\theta-18C_{43}^2}{6C_4}$ &$\frac{\theta^5+3m_3C_{53}^2\theta^3+9C_{53}^2\theta}{3C_5}$ \bigstrut \rule{0pt}{18pt} \\ \cline{2-5}
& B4 & $\frac{\theta^3}{C_3}$  &$\frac{\theta^4+C_4\theta}{2C_4}$ &$\frac{\theta^5+3m_3C_{53}^2\theta^3+9C_{53}^2\theta}{3C_5}$ \bigstrut \rule{0pt}{18pt} \\ \cline{1-5}
\end{longtable}
Table 2:  $\beta, \gamma, \delta$ corresponding to Cases of Table 1.
\end{center}

\section{Computing the shape of pure sextic fields}\label{s 3}

\par In this section, we compute Gram matrices for pure sextic fields, which in turn allows us to parameterize their shapes. Some computations generalize to all pure fields $\Q(\sqrt[n]{m})$. We begin with the case where the parameter $m$ satisfies condition $(A_1, B_1)$. Under this assumption, an integral basis for the field $K = \mathbb{Q}(\theta)$, where $\theta^6 = m$, is given by

$$
\left\{1,\, \theta,\, \frac{\theta^2}{C_2},\, \frac{\theta^3}{C_3},\, \frac{\theta^4}{C_4},\, \frac{\theta^5}{C_5} \right\},
$$
\noindent where the constants $C_i \in \mathbb{Z}_{>0}$ depend explicitly on $m$ (see \eqref{C_k defn}). This choice of basis enables us to explicitly compute the trace pairing and hence the Gram matrix associated with the embedding of $\mathcal{O}_K$ into $\mathbb{R}^6$.


\begin{proposition}\label{gram11}
The Gram matrix for the basis $\left\{ 1,\, \theta,\, \frac{\theta^2}{C_2},\, \frac{\theta^3}{C_3},\, \frac{\theta^4}{C_4},\, \frac{\theta^5}{C_5} \right\}$ is given by:
\[
G_{1,1} = \left\langle J\left(\frac{\theta^{i}}{C_i}\right),\, J\left(\frac{\theta^{j}}{C_j}\right) \right\rangle_{0 \leq i,j \leq 5}
= 6 
\begin{bmatrix}
1&0& 0 & 0 & 0 & 0 \\[8pt]
 0&\frac{\theta^2}{C_1^2}  & 0 & 0 & 0 & 0 \\[8pt]
 0& 0 & \frac{\theta^4}{C_2^2} & 0 & 0 & 0 \\[8pt]
 0&0 & 0 & \frac{m}{C_3^2} & 0 & 0 \\[8pt]
 0&0 & 0 & 0 & \frac{m\theta^2}{C_4^2} & 0 \\[8pt]
 0&0 & 0 & 0 & 0 & \frac{m\theta^{4}}{C_5^2} \\[8pt]
\end{bmatrix}
\]
\end{proposition}

\begin{proof}
We begin by observing that:
\[
J\left( \frac{\theta^i}{C_i} \right)
= \frac{1}{C_i} \left( \theta^i,\, \zeta_6^i \theta^i,\, \zeta_6^{2i} \theta^i,\, \zeta_6^{3i} \theta^i,\, \zeta_6^{4i} \theta^i,\, \zeta_6^{5i} \theta^i \right),
\]
where $\zeta_6$ is a primitive $6$-th root of unity.
Taking the Hermitian inner product of the vectors, the $(i,j)$-entry of the Gram matrix is given by:
\[
\left\langle J\left( \frac{\theta^i}{C_i} \right),\, J\left( \frac{\theta^j}{C_j} \right) \right\rangle
= \frac{\theta^i \theta^j}{C_i C_j} \sum_{k=0}^{5} \zeta_6^{k(i - j)}.
\]
Since $\sum_{k=0}^5 \zeta_6^{k(i-j)} = 6$ if $i = j$ and $0$ otherwise, we conclude:
\[
\left\langle J\left( \frac{\theta^i}{C_i} \right),\, J\left( \frac{\theta^j}{C_j} \right) \right\rangle
= \begin{cases}
6 \left( \frac{\theta^i}{C_i} \right)^2, & \text{if } i = j, \\
0, & \text{otherwise}.
\end{cases}
\]
\end{proof}
\noindent Writing $K = \mathbb{Q}\left(\sqrt[6]{a_1 a_2^2 a_3^3 a_4^4 a_5^5}\right)$, Proposition  \ref{gram11} implies that the matrix:
 \[
\begin{bmatrix}
\left(a_1^2 a_2^4 a_3^6 a_4^8 a_5^{10}\right)^{1/6} & 0 & 0 & 0 & 0 \\[8pt]
0 & \left(a_1^4 a_2^8 a_4^4 a_5^8\right)^{1/6} & 0 & 0 & 0 \\[8pt]
0 & 0 & \left(a_1^6 a_3^6 a_5^6\right)^{1/6} & 0 & 0 \\[8pt]
0 & 0 & 0 & \left(a_1^8 a_2^4 a_4^8 a_5^4\right)^{1/6} & 0 \\[8pt]
0 & 0 & 0 & 0 & \left(a_1^{10} a_2^8 a_3^6 a_4^4 a_5^2\right)^{1/6}
\end{bmatrix}
\]
represents the shape of $K$. On dividing the above matrix by $\prod_{i=1}^5 a_i$, the shape matrix is given by the real matrix: \[\op{sh}_K=\op{diag}\left(\lambda_1,\, \lambda_2,\, \lambda_3,\, \lambda_2^{-1}/ a_3^2,\, \lambda_1^{-1}\right),\] where
$$\lambda_1=\left(\frac{a_4a_5^2}{a_1^2a_2}\right)^{1/3}, \quad \lambda_2=\left(\frac{a_2a_5}{a_1a_3^3a_4}\right)^{1/3}, \quad \lambda_3 = \frac{1}{a_2a_4}$$ are the associated \emph{shape parameters}. We shall denote the above matrix by $\mathbf{Sh}(\lambda_1, \lambda_2, \lambda_3)$. The computation in Proposition \ref{gram11} generalizes to all pure number fields. When $K$ satisfies $(Ai,Bj)$, we say that $m$ is Type $(i,j)$. In this case, there is an explicit transition matrix $C_{i,j}$ such that 
\[\op{sh}_K=C_{i,j}^T \mathbf{Sh}(\lambda_1,\, \lambda_2,\, \lambda_3)C_{i,j}.\] Thus, for the set of $m$ of Type $(i,j)$ the shape distribution is governed by the distribution of shape parameters $(\lambda_1, \lambda_2, \lambda_3)$. The matrices $C_{i,j}$ are computed in this section, however, first we prove a general result.

\begin{proposition}\label{gen}
Let $n>2$ be a natural number and $K=\Q(\sqrt[n]{m})$ and assume that the set $S$ (see \eqref{defn of S}) is empty. The Gram matrix for the basis $\left\{ 1,\, \theta,\, \frac{\theta^2}{C_2},\, \frac{\theta^3}{C_3},\, \ldots,\frac{\theta^{n-1}}{C_{n-1}} \right\}$ is given by:
\[
G_{1,1} = \left\langle J\left(\frac{\theta^{i}}{C_i}\right), J\left(\frac{\theta^{j}}{C_j}\right) \right\rangle_{0 \leq i,j \leq n-1}
= 
n\begin{bmatrix}
1 & 0 & \cdots & 0 \\[8pt]
0 & \frac{\theta^2}{C_1^2} & \cdots & 0 \\[8pt]
\vdots & \vdots & \ddots & \vdots \\[8pt]
0 & 0 & \cdots & \frac{\theta^{2(n-1)}}{C_{n-1}^2}
\end{bmatrix}.
\]
The transition matrix $C_{1,1}$ is simply taken to be the identity.
\end{proposition}
\begin{proof}
   If \( S = \emptyset \), then, by Theorem~\ref{3.1}, the set \( \left\lbrace  \dfrac{\theta^t}{C_t } ~\bigg\vert~ 0 \leq t \leq n - 1 \right\rbrace \) forms an integral basis of \( K \). Using the same process as in Proposition~\ref{gram11}, we obtain the desired Gram matrix in this case.
\end{proof}

\begin{proposition} Let $K$ be as in Proposition \ref{gen}. Then, the shape of $K$ is represented by the $(n-1)\times (n-1)$ real matrix:
    \[
\begin{bmatrix}
\left(\prod\limits_{i=1}^{n-1} a_i^{2 \cdot i - 2n \left\lfloor \frac{1 \cdot i}{n} \right\rfloor}\right)^{1/n} & 0 & \cdots & 0 \\[8pt]
0 & \left(\prod\limits_{i=1}^{n-1} a_i^{2 \cdot 2 \cdot i - 2n \left\lfloor \frac{2 \cdot i}{n} \right\rfloor}\right)^{1/n} & \cdots & 0 \\[8pt]
\vdots & \vdots & \ddots & \vdots \\[8pt]
0 & 0 & \cdots & \left(\prod\limits_{i=1}^{n-1} a_i^{2 \cdot (n-1) \cdot i - 2n \left\lfloor \frac{(n-1) \cdot i}{n} \right\rfloor}\right)^{1/n}
\end{bmatrix}
\]
\end{proposition}

\begin{proof}
    The result is an immediate consequence of Proposition \ref{gen}.
\end{proof}
\par For a fixed integer $i$, define a set
\[
T_i := \left\{\, j \;\middle|\; 1 \leq j \leq n-1 \text{ and } \frac{i j}{n} = \left\lfloor \frac{i j}{n} \right\rfloor \right\}.
\] 

\noindent Multiplying the above matrix by $\prod_{i=1}^{n-1} a_i^{-1}$, we find that 
\[\op{sh}_K=\begin{cases} \op{diag}\left(\lambda_1, \dots, \lambda_i, \dots, \lambda_{(n-1)/2}, \lambda_{(n-1)/2}^{-1}, \dots, \lambda_{i}^{-1}/\prod\limits_{j\in T_i}a_j^2, \dots, \lambda_1^{-1}\right), &\text{ if }n\text{ is odd,}\\[5mm]
\op{diag}\left(\lambda_1, \dots, \lambda_i, \dots,\lambda_{n/2-1}, \lambda_{n/2},\lambda_{n/2-1}^{-1}, \dots, \lambda_{i}^{-1}/\prod\limits_{j\in T_i}a_j^2, \dots, \lambda_1^{-1}\right), &\text{ if }n\text{ is even.}
\end{cases}\]
Here the \emph{shape parameters} are given as follows:
$$\lambda_i:= \left( \prod\limits_{j=1}^{n-1} a_j^{2\cdot i\cdot j -2\cdot n\left\lfloor \frac{i\cdot j}{n}\right\rfloor - n} \right)^{\frac{1}{n}},\quad 1\leq i\leq \left\lfloor\frac{n-1}{2}\right\rfloor.$$
If $n$ is even, then $$\lambda_{\frac{n}{2}} = \frac{1}{a_2a_4\cdots a_{n-2}
}.$$

\begin{proposition}
The Gram matrix for the basis
\[
\mathfrak{B}_{1,2}=\left\{1,\, \theta,\, \frac{\theta^2}{C_2}, \frac{\theta^3}{C_3},\, \frac{\theta^4+mC_4^2\theta^2+C_4^2}{3C_4},\, \frac{\theta^5 + m C_5^2 \theta^3 + C_5^2 \theta}{3 C_5} \right\}
\]
is given by
\[
G_{1,2}=\begin{bmatrix}
1 & 0 & 0 & 0 & \tfrac{C_4}{3} & 0 \\[8pt]
0 & \theta^2 & 0 & 0 & 0 & \tfrac{C_5 \theta^2}{3 } \\[8pt]
0 & 0 & \tfrac{\theta^4}{C_2^2} & 0 & \tfrac{C_4 \theta^4 m}{3 C_2} & 0 \\[8pt]
0 & 0 & 0 & \tfrac{m}{C_3^2} & 0 & \tfrac{C_5 m^2}{3 C_3} \\[8pt]
\tfrac{C_4}{3} & 0 & \tfrac{C_4 \theta^4 m}{3 C_2} & 0 & \tfrac{C_4^4 \theta^4 m^2 + C_4^4 + \theta^2 m}{9 C_4^2} & 0 \\[8pt]
0 & \tfrac{C_5 \theta^2}{3 } & 0 & \tfrac{C_5 m^2}{3 C_3} & 0 & \tfrac{C_5^2 m^3}{9} + \tfrac{\theta^4 m}{9 C_5^2} + \tfrac{C_5^2 \theta^2}{9 }
\end{bmatrix}.
\]
On the other hand, the transition matrix is as follows:
\[
C_{1,2} = 
\begin{bmatrix}
1 & 0 & 0 & 0 & \tfrac{C_4}{3} & 0 \\[8pt]
0 & 1 & 0 & 0 & 0 & \tfrac{C_5}{3} \\[8pt]
0 & 0 & 1 & 0 & \tfrac{C_2 C_4 m}{3} & 0 \\[8pt]
0 & 0 & 0 & 1 & 0 & \tfrac{C_3 C_5 m}{3} \\[8pt]
0 & 0 & 0 & 0 & \tfrac{1}{3} & 0 \\[8pt]
0 & 0 & 0 & 0 & 0 & \tfrac{1}{3}
\end{bmatrix}
\]
\end{proposition}

\begin{proof}
The transition matrix from the basis 
\[
\left\{ 1,\,\theta,\, \frac{\theta^2}{C_2},\, \frac{\theta^3}{C_3},\, \frac{\theta^4}{C_4},\, \frac{\theta^5}{C_5} \right\}
\quad \text{to} \quad
\left\{1,\, \theta,\, \frac{\theta^2}{C_2},\, \frac{\theta^3}{C_3},\, \frac{\theta^4+mC_4^2\theta^2+C_4^2}{3C_4},\, \frac{\theta^5 + m C_5^2 \theta^3 + C_5^2 \theta}{3 C_5} \right\}
\]
is given by $C_{1,2}$ above. We define:
\[
G_{1,2} := C_{1,2}^{T} \cdot G_{1,1} \cdot C_{1,2}.
\] 
The images of the basis elements under the embedding $J$ are as follows:
\begin{align*}
J(1)&=(1,1,1,1,1,1),\\
J\left(\frac{\theta^{i}}{C_{i}}\right) &= \frac{1}{C_{i}} \left( \theta^i,\, \zeta_6^i \theta^i,\, \zeta_6^{2i} \theta^i,\, \zeta_6^{3i} \theta^i,\, \zeta_6^{4i} \theta^i,\, \zeta_6^{5i} \theta^i \right),~~\text{for}~~  1\leq i\leq 3  \\
&\text{and $\zeta_6$ is a primitive $6$-th root of unity.}\\
J\left(\frac{\theta^4+mC_4^2\theta^2+C_4^2}{3C_4}\right)&=\frac{1}{3C_2C_4}(\theta^4+mC_4^2+C_4^2,  \zeta_6^4\theta^4+mC_4^2\zeta_6^2\theta^2+C_4^2,  \zeta^2\theta^4+mC_4^2\zeta_6^4\theta^4\theta^2+C_4^2, \\ & \quad \theta^4+mC_4^2\theta^2+C_4^2,\zeta_6^4\theta^4+mC_4^2\zeta_6^2\theta^2+C_4^2, \zeta^2\theta^4+mC_4^2\zeta_6^4\theta^4\theta^2+C_4^2 ), \\
J\left(\frac{\theta^5+mC_5^2\theta^3+C_5^2\theta}{3C_5}\right)&=\frac{1}{3C_5}(\theta^5+mC_5^2\theta^3+C_5^2\theta,\zeta_6^5\theta^5+mC_5^2\zeta_6^3\theta^3+C_5^2\zeta_6\theta,  \zeta^4\theta^5+mC_5^2\theta^3+C_5^2\zeta_6^2\theta, \\ & \quad\zeta_6^3\theta^5+mC_5^2\theta^3+C_5^2\zeta_6^3\theta,\zeta_6^2+mC_5^2\theta^3+C_5^2\zeta_6^4\theta , \zeta_6\theta^5+mC_5^2\zeta_6^3\theta^3+C_5^2\zeta_6^5\theta).\\
\end{align*}
Using these expressions, it is straightforward (though computational) to verify that the Gram matrix obtained using the embeddings agrees with the matrix obtained via the formula $G_{1,2} =  C_{1,2}^{T} \cdot G_{1,1} \cdot C_{1,2}.$
\end{proof}
\noindent We compute the Gram matrices $G_{i,j}$ and the transition matrices $C_{i,j}$ in all other cases below.

\begin{proposition}
   The Gram matrix for the integral Basis $\mathfrak{B}_{1, 3}$ is given by $G_{1,3}=$
   \[
\begin{bmatrix}
1 
& 0 
& 0 
& 0 
& \tfrac{3C_{43}^2}{C_4} 
& 0 \\[6pt]

0 
& \theta^2 
& 0 
& \tfrac{2\theta^2 m_3}{ C_3} 
& 0 
& \tfrac{3C_{53}^2 \theta^2}{ C_5} \\[8pt]

0 
& 0 
& \tfrac{\theta^4}{C_2^2} 
& 0 
& \tfrac{C_{43}^2 \theta^4 m_3}{C_2 C_4} 
& 0 \\[8pt]

0 
& \tfrac{2\theta^2 m_3}{ C_3} 
& 0 
& \tfrac{m  + 36\theta^2 m_3^2}{9  C_3^2} 
& 0 
& \tfrac{C_{53}^2 m_3 (m  + 18 C_3 \theta^2)}{3  C_3^2 C_5} \\[8pt]

\tfrac{3C_{43}^2}{C_4} 
& 0 
& \tfrac{C_{43}^2 \theta^4 m_3}{C_2 C_4} 
& 0 
& \tfrac{9 C_{43}^4 \theta^4 m_3^2 + 81 C_{43}^4 + m \theta^2}{9 C_4^2} 
& 0 \\[8pt]

0 
& \tfrac{3C_{53}^2 \theta^2}{ C_5} 
& 0 
& \tfrac{C_{53}^2 m_3 (m  + 18 C_3 \theta^2)}{3  C_3^2 C_5} 
& 0 
& \tfrac{m  C_3^2 \theta^4 + 9 m  C_{53}^4 m_3^2 + 81 C_3^2 C_{53}^4 \theta^2}{9  C_3^2 C_5^2}
\end{bmatrix}
\]

 and the transition matrix corresponding to the above basis is as follows:
 \[
C_{1,3}=\begin{bmatrix}
1 & 0 & 0 & 0 & \tfrac{3C_{43}^2}{C_4} & 0 \\[8pt]
0 & 1 & 0 & \tfrac{2m_3}{C_3} & 0 & \tfrac{3C_{53}^2}{C_5} \\[8pt]
0 & 0 & 1 & 0 & \tfrac{C_2 C_{43}^2 m_3}{C_4} & 0 \\
0 & 0 & 0 & \tfrac{1}{3} & 0 & \tfrac{C_{53}^2 m_3}{C_5} \\[8pt]
0 & 0 & 0 & 0 & \tfrac{1}{3} & 0 \\[8pt]
0 & 0 & 0 & 0 & 0 & \tfrac{1}{3}
\end{bmatrix}
\]
\end{proposition}
\begin{proposition}
    The Gram matrix for the integral basis $\mathfrak{B}_{1,4}$ is given by $G_{1,4}=$
    \[
\begin{bmatrix}
1 & 0 & 0 & 0 & 0 & 0 \\[8pt]
0 & \theta^2 & 0 & 0 & 0 & \tfrac{3C_{53}^2 \theta^2}{ C_5} \\[8pt]
0 & 0 & \tfrac{\theta^4}{C_2^2} & 0 & 0 & 0 \\[8pt]
0 & 0 & 0 & \tfrac{m}{C_3^2} & 0 & \tfrac{C_{53}^2 m m_3}{C_3 C_5} \\[8pt]
0 & 0 & 0 & 0 & \tfrac{\theta^2 m}{C_4^2} & 0 \\[8pt]
0 & \tfrac{3C_{53}^2 \theta^2}{ C_5} & 0 & \tfrac{C_{53}^2 m m_3}{C_3 C_5} & 0 & \tfrac{9 m  C_{53}^4 m_3^2 + m  \theta^4 + 81 C_{53}^4 \theta^2}{9  C_5^2}
\end{bmatrix}
\] and the transition matrix is as follows: 
\[
C_{1,4}=\begin{bmatrix}
1 & 0 & 0 & 0 & 0 & 0 \\[8pt]
0 & 1 & 0 & 0 & 0 & \tfrac{3 C_{53}^2}{C_5} \\[8pt]
0 & 0 & 1 & 0 & 0 & 0 \\[8pt]
0 & 0 & 0 & 1 & 0 & \tfrac{C_3 C_{53}^2 m_3}{C_5} \\[8pt]
0 & 0 & 0 & 0 & 1 & 0 \\[8pt]
0 & 0 & 0 & 0 & 0 & \tfrac{1}{3}
\end{bmatrix}.
\]
\end{proposition}
\begin{proposition}
    The Gram matrix for the integral basis $\mathfrak{B}_{2,1}$ is given by $G_{2,1}=$
    \[
\begin{bmatrix}
1 & 0 & 0 & \tfrac{1}{2} & 0 & 0 \\[8pt]
0 & \theta^2 & 0 & 0 & \tfrac{\theta^2}{2} & 0 \\[8pt]
0 & 0 & \tfrac{\theta^4}{C_2^2} & 0 & 0 & \tfrac{\theta^4}{2C_2} \\[8pt]
\tfrac{1}{2} & 0 & 0 & \tfrac{m+C_3^2}{4C_3^2} & 0 & 0 \\[8pt]
0 & \tfrac{\theta^2}{2} & 0 & 0 & \tfrac{\theta^2(C_4^2+m)}{4}  & 0 \\[8pt]
0 & 0 & \tfrac{\theta^4}{2C_2} & 0 & 0 & \tfrac{\theta^4(C_5^2 + m)}{4C_5^2}
\end{bmatrix}
\]
and the transition matrix is as follows:
\[
C_{2,1}=\begin{bmatrix}
1 & 0 & 0 & \tfrac{1}{2} & 0 & 0 \\[8pt]
0 & 1 & 0 & 0 & \tfrac{1}{2} & 0 \\[8pt]
0 & 0 & 1 & 0 & 0 & \tfrac{C_2}{2} \\[8pt]
0 & 0 & 0 & \tfrac{1}{2} & 0 & 0 \\[8pt]
0 & 0 & 0 & 0 & \tfrac{1}{2} & 0 \\[8pt]
0 & 0 & 0 & 0 & 0 & \tfrac{1}{2}
\end{bmatrix}.
\]
\end{proposition}
\begin{proposition}
    The Gram matrix for integral basis $\mathfrak{B}_{2,2}$ is given by $G_{2,2}=$
    \[
\begin{bmatrix}
1 & 0 & 0 & \tfrac{1}{2} & -\tfrac{C_4}{3} & 0 \\[8pt]

0 & \theta^2 & 0 & 0 & \tfrac{\theta^2}{2 } 
& -\tfrac{C_5 \theta^2}{3 } \\[8pt]

0 
& 0 
& \tfrac{\theta^4}{C_2^2} 
& 0 
& -\tfrac{ C_4 \theta^4 m}{3 C_2} 
& \tfrac{ \theta^4}{2 C_2} \\[8pt]

\tfrac{1}{2} 
& 0 
& 0 
& \tfrac{m+C_3^2}{4 C_3^2}
& -\tfrac{C_4}{6} 
& -\tfrac{ C_5 m^2}{6 C_3} \\[8pt]

-\tfrac{C_4}{3} 
& \tfrac{\theta^2}{2 } 
& -\tfrac{ C_4 \theta^4 m}{3 C_2} 
& -\tfrac{C_4}{6} 
& \tfrac{ C_4^2 \theta^4 m^2}{9 }
  + \tfrac{C_4^2}{9}
  + \tfrac{\theta^2 m}{36 C_4^2}
  + \tfrac{\theta^2}{4 } 
& -\tfrac{ C_4 \theta^4 m}{6 }
  - \tfrac{C_5 \theta^2}{6 } \\[8pt]

0 
& -\tfrac{C_5 \theta^2}{3 } 
& \tfrac{ \theta^4}{2 C_2} 
& -\tfrac{C_5 m^2}{6 C_3} 
& -\tfrac{ C_4 \theta^4 m}{6 }
  - \tfrac{C_5 \theta^2}{6 } 
& \tfrac{ C_5^2 m^3}{9}
  + \tfrac{ \theta^4}{4 }
  + \tfrac{\theta^4 m}{36 C_5^2}
  + \tfrac{C_5^2 \theta^2}{9 }
\end{bmatrix}
\]
 and the transition matrix is as follows:
\[
C_{2,2}=\begin{bmatrix}
1 & 0 & 0 & \tfrac{1}{2} & -\tfrac{C_4}{3} & 0 \\[8pt]
0 & 1 & 0 & 0 & \tfrac{1}{2} & -\tfrac{C_5}{3} \\[8pt]
0 & 0 & 1 & 0 & -\tfrac{C_2 C_4 m}{3} & \tfrac{C_2}{2} \\[8pt]
0 & 0 & 0 & \tfrac{1}{2} & 0 & -\tfrac{C_3 C_5 m}{3} \\[8pt]
0 & 0 & 0 & 0 & \tfrac{1}{6} & 0 \\[8pt]
0 & 0 & 0 & 0 & 0 & \tfrac{1}{6}
\end{bmatrix}.
\]
\end{proposition}
\begin{proposition}
    The Gram matrix for the basis $\mathfrak{B}_{2, 3}$ is given by $G_{2,3}=$
\[
\resizebox{\textwidth}{!}{$
\begin{bmatrix}
1 & 0 & 0 & \tfrac{1}{2} & -\tfrac{3 C_{43}^2}{C_4} & 0 \\[8pt]

0 
& \theta^2 
& 0 
& -\tfrac{2 C_{33}^2 \theta^2 m_3}{ C_3} 
& \tfrac{\theta^2}{2 } 
& -\tfrac{3 C_{53}^2 \theta^2}{ C_5} \\[8pt]

0 
& 0 
& \tfrac{\theta^4}{C_2^2} 
& 0 
& -\tfrac{C_{43}^2 \theta^4 m_3}{C_2 C_4} 
& \tfrac{\theta^4}{2 C_2} \\[8pt]

\tfrac{1}{2} 
& -\tfrac{2 C_{33}^2 \theta^2 m_3}{ C_3} 
& 0 
& \tfrac{m}{36 C_3^2} 
  + \tfrac{4 C_{33}^4 \theta^2 m_3^2}{ C_3^2} 
  + \tfrac{1}{4} 
& -\tfrac{3 C_{43}^2}{2 C_4} 
  - \tfrac{C_{33}^2 \theta^2 m_3}{ C_3} 
& -\tfrac{C_{53}^2 m_3 (m  - 36 C_{33}^2 \theta^2)}{6  C_3 C_5} \\[8pt]

-\tfrac{3 C_{43}^2}{C_4} 
& \tfrac{\theta^2}{2 } 
& -\tfrac{C_{43}^2 \theta^4 m_3}{C_2 C_4} 
& -\tfrac{3 C_{43}^2}{2 C_4} 
  - \tfrac{C_{33}^2 \theta^2 m_3}{ C_3} 
& \tfrac{36  C_{43}^4 \theta^4 m_3^2 + 324  C_{43}^4 + m  \theta^2 + 9 C_4^2 \theta^2}{36  C_4^2} 
& -\tfrac{C_{43}^2 \theta^4 m_3}{2 C_4} 
  - \tfrac{3 C_{53}^2 \theta^2}{2  C_5} \\[8pt]

0 
& -\tfrac{3 C_{53}^2 \theta^2}{ C_5} 
& \tfrac{\theta^4}{2 C_2} 
& -\tfrac{C_{53}^2 m_3 (m  - 36 C_{33}^2 \theta^2)}{6  C_3 C_5} 
& -\tfrac{C_{43}^2 \theta^4 m_3}{2 C_4} 
  - \tfrac{3 C_{53}^2 \theta^2}{2  C_5} 
& \tfrac{9  C_5^2 \theta^4 + 36 m  C_{53}^4 m_3^2 + m  \theta^4 + 324 C_{53}^4 \theta^2}{36  C_5^2}
\end{bmatrix}
$}
\]

and the transition matrix is as follows:
\[
C_{2,3}=\begin{bmatrix}
1 & 0 & 0 & \tfrac{1}{2} & -\tfrac{3 C_{43}^2}{C_4} & 0 \\[8pt]
0 & 1 & 0 & -\tfrac{2 C_{33}^2 m_3}{C_3} & \tfrac{1}{2} & -\tfrac{3 C_{53}^2}{C_5} \\[8pt]
0 & 0 & 1 & 0 & -\tfrac{C_2 C_{43}^2 m_3}{C_4} & \tfrac{C_2}{2} \\[8pt]
0 & 0 & 0 & \tfrac{1}{6} & 0 & -\tfrac{C_3 C_{53}^2 m_3}{C_5} \\[8pt]
0 & 0 & 0 & 0 & \tfrac{1}{6} & 0 \\[8pt]
0 & 0 & 0 & 0 & 0 & \tfrac{1}{6}
\end{bmatrix}.
\]
\end{proposition}
\begin{proposition}
    The Gram matrix for the basis $\mathfrak{B}_{2, 4}$ is given by $G_{2,4}=$
    \[
\begin{bmatrix}
1 & 0 & 0 & \tfrac{1}{2} & 0 & 0 \\[6pt]
0 & \theta^2 & 0 & 0 & \tfrac{\theta^2}{2 } & -\tfrac{3 C_{53}^2 \theta^2}{ C_5} \\[6pt]
0 & 0 & \tfrac{\theta^4}{C_2^2} & 0 & 0 & \tfrac{\theta^4}{2 C_2} \\[6pt]
\tfrac{1}{2} & 0 & 0 & \tfrac{m+C_3^3}{4 C_3^2} & 0 & -\tfrac{C_{53}^2 m m_3}{2 C_3 C_5} \\[6pt]
0 & \tfrac{\theta^2}{2 } & 0 & 0 & \tfrac{\theta^2}{4 } + \tfrac{\theta^2 m}{4 C_4^2} & -\tfrac{3 C_{53}^2 \theta^2}{2  C_5} \\[6pt]
0 & -\tfrac{3 C_{53}^2 \theta^2}{ C_5} & \tfrac{\theta^4}{2 C_2} & -\tfrac{C_{53}^2 m m_3}{2 C_3 C_5} & -\tfrac{3 C_{53}^2 \theta^2}{2  C_5} & \tfrac{9  C_5^2 \theta^4 + 36 m  C_{53}^4 m_3^2 + m  \theta^4 + 324 C_{53}^4 \theta^2}{36  C_5^2}
\end{bmatrix}
\]
and the transition matrix is as follows:
\[
C_{2,4}=\begin{bmatrix}
1 & 0 & 0 & \tfrac{1}{2} & 0 & 0 \\[6pt]
0 & 1 & 0 & 0 & \tfrac{1}{2} & -\tfrac{3 C_{53}^2}{C_5} \\[6pt]
0 & 0 & 1 & 0 & 0 & \tfrac{C_2}{2} \\[6pt]
0 & 0 & 0 & \tfrac{1}{2} & 0 & -\tfrac{C_3 C_{53}^2 m_3}{C_5} \\[6pt]
0 & 0 & 0 & 0 & \tfrac{1}{2} & 0 \\[6pt]
0 & 0 & 0 & 0 & 0 & \tfrac{1}{6}
\end{bmatrix}.
\]
\end{proposition}

\begin{proposition}
    The Gram matrix  for integral basis $\mathfrak{B}_{3, 1}$ is given by $G_{3, 1}=$
    \[
\begin{bmatrix}
1 & 0 & 0 & 0 & 0 & 0 \\[8pt]
0 & \theta^2 & 0 & 0 & 0 & 0 \\[8pt]
0 & 0 & \tfrac{\theta^4}{C_2^2} & 0 & 0 & \tfrac{\theta^4}{2 C_2} \\[8pt]
0 & 0 & 0 & \tfrac{m}{C_3^2} & 0 & 0 \\[8pt]
0 & 0 & 0 & 0 & \tfrac{\theta^2 m}{C_4^2} & 0 \\[8pt]
0 & 0 & \tfrac{\theta^4}{2 C_2} & 0 & 0 & \tfrac{\theta^4 (C_5^2 + m)}{4 C_5^2}
\end{bmatrix}
\] and the transition matrix is as follows:
\[
C_{3,1}=\begin{bmatrix}
1 & 0 & 0 & 0 & 0 & 0 \\[8pt]
0 & 1 & 0 & 0 & 0 & 0 \\[8pt]
0 & 0 & 1 & 0 & 0 & \tfrac{C_2}{2} \\[8pt]
0 & 0 & 0 & 1 & 0 & 0 \\[8pt]
0 & 0 & 0 & 0 & 1 & 0 \\[8pt]
0 & 0 & 0 & 0 & 0 & \tfrac{1}{2}
\end{bmatrix}.
\]
\end{proposition}

\begin{proposition}
   The Gram matrix for the basis $\mathfrak{B}_{3, 2}$ is given by $G_{3,2}=$
    \[
\begin{bmatrix}
1 
& 0 
& 0 
& 0 
& \tfrac{C_4}{3} 
& 0 \\[8pt]

0 
& \theta^2 
& 0 
& 0 
& 0 
& -\tfrac{C_5 \theta^2}{3 } \\[8pt]

0 
& 0 
& \tfrac{\theta^4}{C_2^2} 
& 0 
& \tfrac{ C_4 \theta^4 m}{3 C_2} 
& \tfrac{ \theta^4}{2 C_2} \\[8pt]

0 
& 0 
& 0 
& \tfrac{m}{C_3^2} 
& 0 
& -\tfrac{ C_5 m^2}{3 C_3} \\[6pt]

\tfrac{C_4}{3} 
& 0 
& \tfrac{ C_4 \theta^4 m}{3 C_2} 
& 0 
& \tfrac{ C_4^2 \theta^4 m^2}{9 }
  + \tfrac{C_4^2}{9}
  + \tfrac{\theta^2 m}{9 C_4^2} 
& \tfrac{ C_4 \theta^4 m}{6 } \\[6pt]

0 
& -\tfrac{C_5 \theta^2}{3 } 
& \tfrac{ \theta^4}{2 C_2} 
& -\tfrac{C_5 m^2}{3 C_3} 
& \tfrac{ C_4 \theta^4 m}{6} 
& \tfrac{ C_5^2 m^3}{9 }
  + \tfrac{ \theta^4}{4 }
  + \tfrac{\theta^4 m}{36 C_5^2}
  + \tfrac{C_5^2 \theta^2}{9 }
\end{bmatrix}
\] 
and the transition matrix is as follows:
\[
C_{3,2}=\begin{bmatrix}
1 & 0 & 0 & 0 & \tfrac{C_4}{3} & 0 \\[8pt]
0 & 1 & 0 & 0 & 0 & -\tfrac{C_5}{3} \\[8pt]
0 & 0 & 1 & 0 & \tfrac{C_2 C_4 m}{3} & \tfrac{C_2}{2} \\[8pt]
0 & 0 & 0 & 1 & 0 & -\tfrac{C_3 C_5 m}{3} \\[8pt]
0 & 0 & 0 & 0 & \tfrac{1}{3} & 0 \\[8pt]
0 & 0 & 0 & 0 & 0 & \tfrac{1}{6}
\end{bmatrix}.
\]
\end{proposition}
\begin{proposition} 
 The Gram matrix for the basis $\mathfrak{B}_{3, 3}$ is given by $G_{3,3}=$
\[
\resizebox{\textwidth}{!}{$
\begin{bmatrix}
1 & 0 & 0 & 0 & \tfrac{3 C_{43}^2}{C_4} & 0 \\[8pt]

0 
& \theta^2 
& 0 
& \tfrac{2 C_{33}^2 \theta^2 m_3}{ C_3} 
& 0 
& -\tfrac{3 C_{53}^2 \theta^2}{ C_5} \\[8pt]

0 
& 0 
& \tfrac{\theta^4}{C_2^2} 
& 0 
& \tfrac{C_{43}^2 \theta^4 m_3}{C_2 C_4} 
& \tfrac{\theta^4}{2 C_2} \\[8pt]

0 
& \tfrac{2 C_{33}^2 \theta^2 m_3}{ C_3} 
& 0 
& \tfrac{m  + 36 C_{33}^4 \theta^2 m_3^2}{9  C_3^2} 
& 0 
& -\tfrac{C_{53}^2 m_3 (m  + 18 C_{33}^2 \theta^2)}{3  C_3 C_5} \\[8pt]

\tfrac{3 C_{43}^2}{C_4} 
& 0 
& \tfrac{C_{43}^2 \theta^4 m_3}{C_2 C_4} 
& 0 
& \tfrac{9 C_{43}^4 \theta^4 m_3^2 + 81 C_{43}^4 + m \theta^2}{9 C_4^2} 
& \tfrac{C_{43}^2 \theta^4 m_3}{2 C_4} \\[8pt]

0 
& -\tfrac{3 C_{53}^2 \theta^2}{ C_5} 
& \tfrac{\theta^4}{2 C_2} 
& -\tfrac{C_{53}^2 m_3 (m  + 18 C_{33}^2 \theta^2)}{3  C_3 C_5} 
& \tfrac{C_{43}^2 \theta^4 m_3}{2 C_4} 
& \tfrac{C_{53}^4 m m_3^2}{C_5^2}
  + \tfrac{9 C_{53}^4\theta^2}{ C_5^2}
  + \tfrac{\theta^4(9C_5^2+m)}{36 C_5^2} 
\end{bmatrix}
$}
\]

and the transition matrix is as follows:
\[
C_{3,3}=\begin{bmatrix}
1 & 0 & 0 & 0 & \tfrac{3 C_{43}^2}{C_4} & 0 \\[8pt]
0 & 1 & 0 & \tfrac{2 C_{33}^2 m_3}{C_3} & 0 & -\tfrac{3 C_{53}^2}{C_5} \\[6pt]
0 & 0 & 1 & 0 & \tfrac{C_2 C_{43}^2 m_3}{c_4} & \tfrac{C_2}{2} \\[8pt]
0 & 0 & 0 & \tfrac{1}{3} & 0 & -\tfrac{C_3 C_{53}^2 m_3}{C_5} \\[8pt]
0 & 0 & 0 & 0 & \tfrac{1}{3} & 0 \\[8pt]
0 & 0 & 0 & 0 & 0 & \tfrac{1}{6}
\end{bmatrix}.
\]
\end{proposition}
\begin{proposition}
    The Gram matrix for the basis $\mathfrak{B}_{3,4}$ is given by $G_{3, 4}=$
   \[
\begin{bmatrix}
1 
& 0 
& 0 
& 0 
& 0 
& 0 \\[8pt]

0 
& \theta^2 
& 0 
& 0 
& 0 
& -\tfrac{3 C_{53}^2 \theta^2}{ C_5} \\[8pt]

0 
& 0 
& \tfrac{\theta^4}{C_2^2} 
& 0 
& 0 
& \tfrac{ \theta^4}{2 C_2} \\[8pt]

0 
& 0 
& 0 
& \tfrac{m}{C_3^2} 
& 0 
& -\tfrac{ C_{53}^2 m m_3}{C_3 C_5} \\[8pt]

0 
& 0 
& 0 
& 0 
& \tfrac{\theta^2 m}{C_4^2} 
& 0 \\[8pt]

0 
& -\tfrac{3 C_{53}^2 \theta^2}{ C_5} 
& \tfrac{ \theta^4}{2 C_2} 
& -\tfrac{ C_{53}^2 m m_3}{C_3 C_5} 
& 0 
& \tfrac{ C_{53}^4 m m_3^2}{ C_5^2}
  + \tfrac{ \theta^4}{4 }
  + \tfrac{\theta^4 m}{36 C_5^2}
  + \tfrac{9 C_{53}^4 \theta^2}{ C_5^2}
\end{bmatrix}
\]
 and the transition matrix is as follows:
\[
C_{3,4}=\begin{bmatrix}
1 & 0 & 0 & 0 & 0 & 0 \\[8pt]
0 & 1 & 0 & 0 & 0 & -\tfrac{3C_{53}^2}{C_5} \\[8pt]
0 & 0 & 1 & 0 & 0 & \tfrac{C_2}{2} \\[8pt]
0 & 0 & 0 & 1 & 0 & -\tfrac{C_3 C_{53}^2 m_3}{C_5} \\[8pt]
0 & 0 & 0 & 0 & 1 & 0 \\
0 & 0 & 0 & 0 & 0 & \tfrac{1}{6}
\end{bmatrix}.
\]

\end{proposition}

\begin{proposition}
    The Gram matrix for the basis $\mathfrak{B}_{4, 1}$ is given by $G_{4,1}=$
    \[
\begin{bmatrix}
1 
& 0 
& 0 
& \tfrac{1}{2} 
& 0 
& 0 \\[8pt]

0 
& \theta^2 
& 0 
& 0 
& \tfrac{\theta^2}{2 } 
& 0 \\[8pt]

0 
& 0 
& \tfrac{\theta^4}{C_2^2} 
& 0 
& 0 
& \tfrac{2 C_{52} \theta^4}{C_2 C_5} \\[8pt]

\tfrac{1}{2} 
& 0 
& 0 
& \tfrac{m+C_3^2}{4 C_3^2}
& 0 
& 0 \\[8pt]

0 
& \tfrac{\theta^2}{2 } 
& 0 
& 0 
& \tfrac{\theta^2 m+\theta^2C_4^2}{4C_2^2 C_4^2}
& 0 \\[8pt]

0 
& 0 
& \tfrac{2 C_{52} \theta^4}{C_2 C_5} 
& 0 
& 0 
& \tfrac{(16 C_{52}^2+m) \theta^4}{4C_5^2} 
\end{bmatrix}
\]

and the transition matrix is as follows: 
\[
C_{4,1}=\begin{bmatrix}
1 & 0 & 0 &  \tfrac{1}{2} & 0 & 0 \\[8pt]
0 & 1 & 0 & 0 & \tfrac{1}{2} & 0 \\[8pt]
0 & 0 & 1 & 0 & 0 & \tfrac{2C_2 C_{52}}{C_5} \\[8pt]
0 & 0 & 0 &  \tfrac{1}{2} & 0 & 0 \\[8pt]
0 & 0 & 0 & 0 & \tfrac{1}{2} & 0 \\[8pt]
0 & 0 & 0 & 0 & 0 & \tfrac{1}{2}
\end{bmatrix}.
\]
\end{proposition}

\begin{proposition}
    The Gram matrix for the basis $\mathfrak{B}_{4,2}$ is given by $G_{4, 2}=$
    \[
\begin{bmatrix}
1 
& 0 
& 0 
& \tfrac12 
& -\tfrac{C_4}{3} 
& 0 \\[8pt]

0 
& \theta^2 
& 0 
& 0 
& \tfrac{\theta^2}{2 } 
& -\tfrac{C_5 \theta^2}{3 } \\[8pt]

0 
& 0 
& \tfrac{\theta^4}{C_2^2} 
& 0 
& -\tfrac{C_4 \theta^4 m}{3 C_2} 
& \tfrac{2 C_{52} \theta^4}{C_2 C_5} \\[8pt]

\tfrac12 
& 0 
& 0 
& \tfrac{m+C_3^2}{4C_3^2}
& -\tfrac{C_4}{6} 
& -\tfrac{C_5 m^2}{6 C_3} \\[8pt]

-\tfrac{C_4}{3} 
& \tfrac{\theta^2}{2 } 
& -\tfrac{C_4 \theta^4 m}{3 C_2} 
& -\tfrac{C_4}{6} 
& \tfrac{C_4^2(1+\theta^4 m^2)}{9}
  + \tfrac{\theta^2}{4 }
  + \tfrac{\theta^2 m}{36 C_4^2} 
& -\tfrac{2 C_4 C_{52} \theta^4 m}{3 C_5}
  - \tfrac{C_5 \theta^2}{6 } \\[8pt]

0 
& -\tfrac{C_5 \theta^2}{3 } 
& \tfrac{2 C_{52} \theta^4}{C_2 C_5} 
& -\tfrac{C_5 m^2}{6 C_3} 
& -\tfrac{2 C_4 C_{52} \theta^4 m}{3 C_5}
  - \tfrac{C_5 \theta^2}{6 } 
& \tfrac{C_5^2}{9}\!\left(m^3+\theta^2\right)
  + \tfrac{4 C_{52}^2 \theta^4}{C_5^2}
  + \tfrac{\theta^4 m}{36 C_5^2}
\end{bmatrix}
\]

and the transition matrix is as follows:
\[C_{4,2}=
\begin{bmatrix}
1 & 0 & 0 & \tfrac{1}{2} & -\tfrac{C_4}{3} & 0 \\[8pt]
0 & 1 & 0 & 0 & \tfrac{1}{2} & -\tfrac{C_5}{3} \\[8pt]
0 & 0 & 1 & 0 & -\tfrac{C_2 C_4 m}{3} & \tfrac{2 C_2 C_{52}}{C_5} \\[8pt]
0 & 0 & 0 & \tfrac{1}{2} & 0 & -\tfrac{C_3 C_5 m}{3} \\[8pt]
0 & 0 & 0 & 0 & \tfrac{1}{6} & 0 \\[8pt]
0 & 0 & 0 & 0 & 0 & \tfrac{1}{6}
\end{bmatrix}.
\]
\end{proposition}
\begin{proposition}
    The Gram matrix for  basis $\mathfrak{B}_{4, 3}$ is given by $G_{4,3}=$
    \[
\resizebox{\textwidth}{!}{$
\begin{bmatrix}
1 & 0 & 0 & \tfrac{1}{2} & -\tfrac{3 C_{43}^2}{C_4} & 0 \\[8pt]

0 & \theta^2 & 0 
& -\tfrac{2 C_{33}^2 \theta^2 m_3}{ C_3} 
& \tfrac{\theta^2}{2 } 
& -\tfrac{3 C_{53}^2 \theta^2}{ C_5} \\[8pt]

0 & 0 & \tfrac{\theta^4}{C_2^2} & 0 
& -\tfrac{C_{43}^2 \theta^4 m_3}{C_2 C_4} 
& \tfrac{2 C_{52} \theta^4}{C_2 C_5} \\[8pt]

\tfrac{1}{2} 
& -\tfrac{2 C_{33}^2 \theta^2 m_3}{ C_3} 
& 0 
& \tfrac{m}{36 C_3^2} + \tfrac{4 C_{33}^4 \theta^2 m_3^2}{ C_3^2} + \tfrac{1}{4} 
& -\tfrac{3 C_{43}^2}{2 C_4} - \tfrac{C_{33}^2 \theta^2 m_3}{ C_3} 
& -\tfrac{C_{53}^2 m_3 (m  - 36 C_{33}^2 \theta^2)}{6  C_3 C_5} \\[8pt]

-\tfrac{3 C_{43}^2}{C_4} 
& \tfrac{\theta^2}{2 } 
& -\tfrac{C_{43}^2 \theta^4 m_3}{C_2 C_4} 
& -\tfrac{3 C_{43}^2}{2 C_4} - \tfrac{C_{33}^2 \theta^2 m_3}{ C_3} 
& \tfrac{36  C_{43}^4 \theta^4 m_3^2 + 324  C_{43}^4 + m  \theta^2 + 9 C_4^2 \theta^2}{36  C_4^2} 
& -\tfrac{\theta^2 (4 C_{52} m_3  C_{43}^2 \theta^2 + 3 C_4 C_{53}^2)}{2  C_4 C_5} \\[8pt]

0 
& -\tfrac{3 C_{53}^2 \theta^2}{ C_5} 
& \tfrac{2 C_{52} \theta^4}{C_2 C_5} 
& -\tfrac{C_{53}^2 m_3 (m  - 36 C_{33}^2 \theta^2)}{6  C_3 C_5} 
& -\tfrac{\theta^2 (4 C_{52} m_3  C_{43}^2 \theta^2 + 3 C_4 C_{53}^2)}{2  C_4 C_5} 
& \tfrac{144  C_{52}^2 \theta^4 + 36 m  C_{53}^4 m_3^2 + m  \theta^4 + 324 C_{53}^4 \theta^2}{36  C_5^2}
\end{bmatrix}
$}
\] and the transition matrix  as follows:
\[C_{4,3}=
\begin{bmatrix}
1 & 0 & 0 & \tfrac{1}{2} & -\tfrac{3 C_{43}^2}{C_4} & 0 \\[8pt]
0 & 1 & 0 & -\tfrac{2 C_{33}^2 m_3}{C_3} & \tfrac{1}{2} & -\tfrac{3 C_{53}^2}{C_5} \\[8pt]
0 & 0 & 1 & 0 & -\tfrac{C_2 C_{43}^2 m_3}{C_4} & \tfrac{2 C_2 C_{52}}{C_5} \\[8pt]
0 & 0 & 0 & \tfrac{1}{6} & 0 & -\tfrac{C_3 C_{53}^2 m_3}{C_5} \\[8pt]
0 & 0 & 0 & 0 & \tfrac{1}{6} & 0 \\[8pt]
0 & 0 & 0 & 0 & 0 & \tfrac{1}{6}
\end{bmatrix}.
\]
\end{proposition}
\begin{proposition}
    The Gram matrix for the basis $\mathfrak{B}_{4, 4}$ is given by $G_{4,4}=$ 
    \[
\begin{bmatrix}
1 & 0 & 0 & \tfrac{1}{2} & 0 & 0 \\[8pt]

0 
& \theta^2 
& 0 
& 0 
& \tfrac{\theta^2}{2 } 
& -\tfrac{3 C_{53}^2 \theta^2}{ C_5} \\[8pt]

0 
& 0 
& \tfrac{\theta^4}{C_2^2} 
& 0 
& 0 
& \tfrac{2 C_{52} \theta^4}{C_2 C_5} \\[8pt]

\tfrac{1}{2} 
& 0 
& 0 
& \tfrac{m}{4 C_3^2} + \tfrac{1}{4} 
& 0 
& -\tfrac{C_{53}^2 m m_3}{2 C_3 C_5} \\[8pt]

0 
& \tfrac{\theta^2}{2 } 
& 0 
& 0 
& \tfrac{\theta^2}{4 } + \tfrac{\theta^2 m}{4 C_4^2} 
& -\tfrac{3 C_{53}^2 \theta^2}{2  C_5} \\[8pt]

0 
& -\tfrac{3 C_{53}^2 \theta^2}{ C_5} 
& \tfrac{2 C_{52} \theta^4}{C_2 C_5} 
& -\tfrac{C_{53}^2 m m_3}{2 C_3 C_5} 
& -\tfrac{3 C_{53}^2 \theta^2}{2  C_5} 
& \tfrac{144  C_{52}^2 \theta^4 + 36 m  C_{53}^4 m_3^2 + m  \theta^4 + 324 C_{53}^4 \theta^2}{36  C_5^2}
\end{bmatrix}
\]
and the transition matrix is as follows:

\[C_{4,4}=
\begin{bmatrix}
1 & 0 & 0 & \tfrac{1}{2} & 0 & 0 \\[8pt]
0 & 1 & 0 & 0 & \tfrac{1}{2} & -\tfrac{3 C_{53}^2}{C_5} \\[8pt]
0 & 0 & 1 & 0 & 0 & \tfrac{2 C_2 C_{52}}{C_5} \\[8pt]
0 & 0 & 0 & \tfrac{1}{2} & 0 & -\tfrac{C_3 C_{53}^2 m_3}{C_5} \\[8pt]
0 & 0 & 0 & 0 & \tfrac{1}{2} & 0 \\[8pt]
0 & 0 & 0 & 0 & 0 & \tfrac{1}{6}
\end{bmatrix}.
\]
\end{proposition}
\begin{proposition}
    The Gram matrix for the basis $\mathfrak{B}_{5,1}$ is given by $G_{5,1}=$
    \[
\begin{bmatrix}
1 & 0 & 0 & 0 & 0 & 0 \\[8pt]

0 & \theta^2 & 0 & 0 & \tfrac{\theta^2}{2 } & 0 \\[8pt]

0 & 0 & \tfrac{\theta^4}{C_2^2} & 0 & 0 & 0 \\[8pt]

0 & 0 & 0 & \tfrac{m}{C_3^2} & 0 & 0 \\[8pt]

0 & \tfrac{\theta^2}{2 } & 0 & 0 
& \tfrac{\theta^2}{4 } + \tfrac{\theta^2 m}{4 C_4^2} & 0 \\[8pt]

0 & 0 & 0 & 0 & 0 & \tfrac{\theta^4 m}{C_5^2}
\end{bmatrix}
\] and the transition matrix  as follows: 
\[C_{5,1}=
\begin{bmatrix}
1 & 0 & 0 & 0 & 0   & 0 \\[8pt]
0 & 1 & 0 & 0 & \tfrac{1}{2} & 0 \\[8pt]
0 & 0 & 1 & 0 & 0   & 0 \\[8pt]
0 & 0 & 0 & 1 & 0   & 0 \\[8pt]
0 & 0 & 0 & 0 & \tfrac{1}{2} & 0 \\[8pt]
0 & 0 & 0 & 0 & 0   & 1
\end{bmatrix}.
\]
\end{proposition}
\begin{proposition}
    The Gram matrix for the basis $\mathfrak{B}_{5, 2}$ is given by $G_{5,2}=$
    \[
\begin{bmatrix}
1
& 0
& 0
& 0
& -\tfrac{C_4}{3}
& 0 \\[8pt]

0
& \theta^2
& 0
& 0
& \tfrac{\theta^2}{2 }
& \tfrac{C_5 \theta^2}{3 } \\[8pt]

0
& 0
& \tfrac{\theta^4}{C_2^2}
& 0
& -\tfrac{C_4 \theta^4 m}{3 C_2}
& 0 \\[8pt]

0
& 0
& 0
& \tfrac{m}{C_3^2}
& 0
& \tfrac{C_5 m^2}{3 C_3} \\[8pt]

-\tfrac{C_4}{3}
& \tfrac{\theta^2}{2 }
& -\tfrac{C_4 \theta^4 m}{3 C_2}
& 0
& \tfrac{C_4^2 \theta^4 m^2}{9}
  + \tfrac{C_4^2}{9}
  + \tfrac{\theta^2 m}{36 C_4^2}
  + \tfrac{\theta^2}{4 }
& \tfrac{C_5 \theta^2}{6 } \\[8pt]

0
& \tfrac{C_5 \theta^2}{3 }
& 0
& \tfrac{C_5 m^2}{3 C_3}
& \tfrac{C_5 \theta^2}{6 }
& \tfrac{C_5^2 m^3}{9}
  + \tfrac{\theta^4 m}{9 C_5^2}
  + \tfrac{C_5^2 \theta^2}{9 }
\end{bmatrix}
\] 
and the transition matrix  as follows:
\[C_{5,2}=
\begin{bmatrix}
1 & 0 & 0 & 0 & -\tfrac{C_4}{3} & 0 \\[8pt]
0 & 1 & 0 & 0 & \frac{1}{2} & \frac{C_5}{3} \\[8pt]
0 & 0 & 1 & 0 & -\frac{C_2 C_4 m}{3} & 0 \\[8pt]
0 & 0 & 0 & 1 & 0 & \frac{C_3 C_5 m}{3} \\[8pt]
0 & 0 & 0 & 0 & \frac{1}{6} & 0 \\[8pt]
0 & 0 & 0 & 0 & 0 & \frac{1}{3}
\end{bmatrix}.
\]
\end{proposition}

\begin{proposition}
    The Gram matrix for the basis $\mathfrak{B}_{5, 3}$ is given by $G_{5,3}=$
  \[
\resizebox{\textwidth}{!}{$
\begin{bmatrix}
1 & 0 & 0 & 0 & -\tfrac{3 C_{43}^2}{C_4} & 0 \\[8pt]
0 & \theta^2 & 0 & \tfrac{2 C_{33}^2 \theta^2 m_3}{ C_3} & \tfrac{\theta^2}{2 } & \tfrac{3 C_{53}^2 \theta^2}{ C_5} \\[8pt]
0 & 0 & \tfrac{\theta^4}{C_2^2} & 0 & -\tfrac{C_{43}^2 \theta^4 m_3}{C_2 C_4} & 0 \\[8pt]
0 & \tfrac{2 C_{33}^2 \theta^2 m_3}{ C_3} & 0 & \tfrac{m  + 36 C_{33}^4 \theta^2 m_3^2}{9  C_3^2} & \tfrac{C_{33}^2 \theta^2 m_3}{ C_3} & \tfrac{C_{53}^2 m_3 (m  + 18 C_{33}^2 \theta^2)}{3  C_3 C_5} \\[8pt]
-\tfrac{3 C_{43}^2}{C_4} & \tfrac{\theta^2}{2 } & -\tfrac{C_{43}^2 \theta^4 m_3}{C_2 C_4} & \tfrac{C_{33}^2 \theta^2 m_3}{ C_3} & \tfrac{36  C_{43}^4 \theta^4 m_3^2 + 324  C_{43}^4 + m  \theta^2 + 9 C_4^2 \theta^2}{36  C_4^2} & \tfrac{3 C_{53}^2 \theta^2}{2  C_5} \\[8pt]
0 & \tfrac{3 C_{53}^2 \theta^2}{ C_5} & 0 & \tfrac{C_{53}^2 m_3 (m  + 18 C_{33}^2 \theta^2)}{3  C_3 C_5} & \tfrac{3 C_{53}^2 \theta^2}{2  C_5} & \tfrac{9 m  C_{53}^4 m_3^2 + m  \theta^4 + 81 C_{53}^4 \theta^2}{9  C_5^2}
\end{bmatrix}
$}
\]
and the transition matrix as follows:
\[C_{5,3}=
\begin{bmatrix}
1 & 0 & 0 & 0 & -\tfrac{3 C_{43}^2}{C_4} & 0 \\[8pt]
0 & 1 & 0 & \tfrac{2 C_{33}^2 m_3}{C_3} & \tfrac{1}{2} & \tfrac{3 C_{53}^2}{C_5} \\[8pt]
0 & 0 & 1 & 0 & -\tfrac{C_2 C_{43}^2 m_3}{C_4} & 0 \\[8pt]
0 & 0 & 0 & \tfrac{1}{3} & 0 & \tfrac{C_3 C_{53}^2 m_3}{C_5} \\[8pt]
0 & 0 & 0 & 0 & \tfrac{1}{6} & 0 \\[8pt]
0 & 0 & 0 & 0 & 0 & \tfrac{1}{3}
\end{bmatrix}.
\]
\end{proposition}
\begin{proposition}
    The Gram matrix for the basis $\mathfrak{B}_{5, 4}$ is given by $G_{5,4}=$
    \[
\begin{bmatrix}
1 & 0 & 0 & 0 & 0 & 0 \\[8pt]

0 
& \theta^2 
& 0 
& 0 
& \tfrac{\theta^2}{2 } 
& \tfrac{3 C_{53}^2 \theta^2}{ C_5} \\[8pt]

0 
& 0 
& \tfrac{\theta^4}{C_2^2} 
& 0 
& 0 
& 0 \\[8pt]

0 
& 0 
& 0 
& \tfrac{m}{C_3^2} 
& 0 
& \tfrac{C_{53}^2 m m_3}{C_3 C_5} \\[8pt]

0 
& \tfrac{\theta^2}{2 } 
& 0 
& 0 
& \tfrac{\theta^2(C_4^2+m)}{4 C_4^2} 
& \tfrac{3 C_{53}^2 \theta^2}{2  C_5} \\[8pt]

0 
& \tfrac{3 C_{53}^2 \theta^2}{ C_5} 
& 0 
& \tfrac{C_{53}^2 m m_3}{C_3 C_5} 
& \tfrac{3 C_{53}^2 \theta^2}{2  C_5} 
& \tfrac{9 m  C_{53}^4 m_3^2 + m  \theta^4 + 81 C_{53}^4 \theta^2}{9  C_5^2}
\end{bmatrix}
\]
and the transition matrix is as follows:
\[C_{5,4}=
\begin{bmatrix}
1 & 0 & 0 & 0 & 0 & 0 \\[8pt]
0 & 1 & 0 & 0 & \tfrac{1}{2} & \tfrac{3 C_{53}^2}{C_5} \\[8pt]
0 & 0 & 1 & 0 & 0 & 0 \\[8pt]
0 & 0 & 0 & 1 & 0 & \tfrac{C_3 C_{53}^2 m_3}{C_5} \\[8pt]
0 & 0 & 0 & 0 & \tfrac{1}{2} & 0 \\[8pt]
0 & 0 & 0 & 0 & 0 & \tfrac{1}{3}
\end{bmatrix}.
\]
\end{proposition}

\section{Equidistribution results for \texorpdfstring{$(\lambda_1, \lambda_2, \lambda_3)$}{(lambda1, lambda2, lambda3)}}\label{s 4}
\par
In this section, we prove the first main result of the article regarding the distribution of the vector $(\lambda_1, \lambda_2, \lambda_3)$. Fix a Type $(i,j)$ and a sign $\epsilon\in\{+,-\}$. Recall that a sixth-power-free integer $m$ is of Type $(i,j)$ if the field $\Q(\sqrt[6]{m})$ satisfies $(Ai)$ and $(Bj)$. We shall consider $m$ of sign $\epsilon$ and of Type $(i,j)$. Let $a=(a_1,\dots,a_5)\in\mathbb{Z}^5$ be the strongly carefree couple such that $m=\epsilon a_1 a_2^2 a_3^3 a_4^4 a_5^5$. Associated to $a$ are the shape parameters
\[
\lambda(m)=(\lambda_1(m),\,\lambda_2(m),\,\lambda_3(m))
=
\left(
\left(\frac{a_4 a_5^2}{a_1^2 a_2}\right)^{1/3},\,
\left(\frac{a_2 a_5}{a_1 a_3^3 a_4}\right)^{1/3},\,
\frac{1}{a_2 a_4}
\right).
\]
\noindent Observe that replacing $a$ by $a^\vee := (a_5,a_4,a_3,a_2,a_1)$ leaves the
associated field invariant. We may therefore
assume, without loss of generality, that $a_4a_5^2\geq a_1^2a_2$. Fix a rectangular region
\[
\mathscr{R}:=[R_1',R_1]\times [R_2',R_2]\times [R_3',R_3]
\subset
[1, \infty)\times (0,\infty)\times [1,\infty).
\]
For $N>0$, let
$\mathcal{C}(N,\mathscr{R})$
denote the set of integer tuples $a=(a_1,\dots,a_5)$ satisfying the following conditions:
\begin{itemize}
    \item $a_i>0$ for all $i$,
    \item $a_1^5 a_2^4 a_3^3 a_4^4 a_5^5 \le N$,
    \item $(\lambda_1(m)^3, \lambda_2(m)^3, \lambda_3(m)^{-1})\in\mathscr{R}$, equivalently,
    \[
    R_1'\le \frac{a_4 a_5^2}{a_1^2 a_2}\le R_1,\qquad
    R_2'\le \frac{a_2 a_5}{a_1 a_3^3 a_4}\le R_2,\qquad
    a_2 a_4 \in [R_3',R_3].
    \]
\end{itemize}
Since $a_2 a_4$ is a positive integer, we may assume without loss of
generality that $R_3'$ and $R_3$ are integers. Our first goal is to estimate
$\#\mathcal{C}(N,\mathscr{R})$ as a function of $N$ and the
parameters $R_i,R_i'$. We then impose the squarefreeness and coprimality
conditions on the integer tuples $(a_1,\dots, a_5)$ via sieve methods, yielding a formula for the limit
\[\lim_{N\rightarrow \infty} \frac{\#\{K\mid m\text{ has sign }\epsilon \text{ of Type }(i,j)\text{ with }(\lambda_1(m)^3, \lambda_2(m)^3, \lambda_3(m)^{-1})\in \mathscr{R}\text{, and }|\Delta_K|\leq 2^{v_1}3^{v_2}N\}}{N^{1/5}},\] where we recall that $v_1$ and $v_2$ are given according to \eqref{v_1 v_2}.
\begin{lemma}
Let $0\le L_1'<L_1$ and $0\le L_2'<L_2$ be real numbers, and let $N>0$.   
Let $\mathcal{M}(N,L_1',L_1,L_2',L_2)$ denote the region of points
$(x_1,x_3,x_5)\in\mathbb{R}^3$ satisfying
\begin{enumerate}
    \item $x_1^5 x_3^3 x_5^5 \le N$,
    \item $x_1,x_3,x_5>0$,
    \item $\dfrac{x_5}{x_1}\in[L_1',L_1]$,
    \item $\dfrac{x_5}{x_3^3 x_1}\in[L_2',L_2]$.
\end{enumerate}
Then the volume of this region given by
\[
V(N,L_1',L_1,L_2',L_2)
:=\op{vol}\left(\mathcal{M}(N,L_1',L_1,L_2',L_2)\right)=
\frac{75}{8}\,N^{1/5}
\bigl(L_1^{2/15}-L_1'^{2/15}\bigr)
\bigl(L_2'^{-2/15}-L_2^{-2/15}\bigr).
\]
\end{lemma}

\begin{proof}
Define new variables $u,v$ by
\[
u=\frac{x_5}{x_1} \qquad \text{and}\quad v=\frac{x_5}{x_3^3 x_1}.
\]
Solving for $(x_3,x_5)$ in terms of $(x_1,u,v)$ gives
\[
x_5=u x_1 \quad \text{and}\quad x_3=\left(\frac{u}{v}\right)^{1/3}.
\]
In these coordinates, the constraints become simply $u\in[L_1',L_1]$ and $v\in[L_2',L_2]$. Using the above expressions, we compute
\[
x_1^5 x_3^3 x_5^5
=
x_1^5 \cdot \frac{u}{v} \cdot (u x_1)^5
=
x_1^{10}\frac{u^6}{v}.
\]
Thus the condition $x_1^5 x_3^3 x_5^5\le N$ is equivalent to
\[
x_1^{10}\le \frac{Nv}{u^6},
\]
or equivalently,
\[
0<x_1\le \left(\frac{Nv}{u^6}\right)^{1/10}.
\]
\noindent We compute the Jacobian determinant of the map
\[
(x_1,u,v)\longmapsto (x_1,x_3,x_5).
\]
A direct calculation yields
\[
\frac{\partial(x_1,x_3,x_5)}{\partial(x_1,u,v)}
=
\begin{vmatrix}
1 & 0 & 0 \\[8pt]
0 & \frac{1}{3}u^{-2/3}v^{-1/3} & -\frac{1}{3}u^{1/3}v^{-4/3} \\[8pt]
u & x_1 & 0
\end{vmatrix}
=
\frac{u^{1/3}x_1}{3v^{4/3}}.
\]
The volume is therefore
\[
V(N,L_1',L_1,L_2',L_2)
=
\int_{L_1'}^{L_1}
\int_{L_2'}^{L_2}
\int_{0}^{(Nv/u^6)^{1/10}}
\frac{u^{1/3}x_1}{3v^{4/3}}
\,dx_1\,dv\,du.
\]
Integrating with respect to $x_1$ first, we obtain
\[
\int_0^{(Nv/u^6)^{1/10}} x_1\,dx_1
=
\frac12\left(\frac{Nv}{u^6}\right)^{1/5}.
\]
Hence
\[
V(N,L_1',L_1,L_2',L_2)
=
\frac{N^{1/5}}{6}
\int_{L_1'}^{L_1} u^{1/3-6/5}\,du
\int_{L_2'}^{L_2} v^{1/5-4/3}\,dv.
\]
Simplifying the exponents gives
\[
V(N,L_1',L_1,L_2',L_2)
=
\frac{N^{1/5}}{6}
\int_{L_1'}^{L_1} u^{-13/15}\,du
\int_{L_2'}^{L_2} v^{-17/15}\,dv.
\]
A straightforward computation yields
\[
\int u^{-13/15}\,du=\frac{15}{2}u^{2/15},
\qquad
\int v^{-17/15}\,dv=-\frac{15}{2}v^{-2/15}.
\]
Substituting the bounds of integration, we obtain
\[
V(N,L_1',L_1,L_2',L_2)
=
\frac{75}{8}\,N^{1/5}
\bigl(L_1^{2/15}-L_1'^{2/15}\bigr)
\bigl(L_2'^{-2/15}-L_2^{-2/15}\bigr),
\]
as claimed.
\end{proof}

\par
In order to prove estimates for lattice points in a region, we shall appeal to the Lipchitz principle, which gives a uniform bound for the
discrepancy between lattice points and area. Let $\mathcal{R}\subset\mathbb{R}^2$ be a closed, bounded region whose
boundary is rectifiable with total length $L$. 

\begin{lemma}[Lipschitz principle]\label{davenport}
One has
\[
\bigl|\#(\mathcal{R}\cap\mathbb{Z}^2)-\op{Area}(\mathcal{R})\bigr|
\le 4(L+1).
\]
\end{lemma}

\begin{lemma}\label{hashtag N lemma}
Let $1\le L_1'<L_1$ and $0\le L_2'<L_2$ be real numbers, and let $N>0$.
Define
\[
\mathcal{N}(N,L_1',L_1,L_2',L_2)
:=\mathcal{M}(N,L_1',L_1,L_2',L_2)\cap \mathbb{Z}^3.
\]
Then one has the asymptotic estimate
\[
\#\mathcal{N}(N,L_1',L_1,L_2',L_2)
=
V(N,L_1',L_1,L_2',L_2)
+
O\!\left(N^{1/10}\right),
\]
where the implied constant depends on $L_1',L_1,L_2',L_2$.
\end{lemma}

\begin{proof}
For a fixed integer $x_1\ge 1$, consider the region $\mathcal{C}_{x_1}$ of $(x_3,x_5)\in\mathbb{R}_{>0}^2$
such that $(x_1,x_3,x_5)\in\mathcal{M}(N,L_1',L_1,L_2',L_2)$.
In the $(u,v)$–coordinates $u=\frac{x_5}{x_1}$ and $v=\frac{x_5}{x_3^3 x_1}$, this region is given by
\[
\mathcal{C}_{x_1}=\left\{(u,v)\mid L_1'\le u\le L_1,\quad
L_2'\le v\le L_2,\quad\text{and}\quad
v\ge \frac{x_1^{10}u^6}{N}\right\}.\]
By Lemma \ref{davenport}, the number of integer points with a fixed value of $x_1$ is
\[
\#\{(x_3,x_5)\in\mathbb{Z}_{>0}^2 : (x_1,x_3,x_5)\in\mathcal{M}\}
=
\op{area}(\mathcal{C}_{x_1})
+
O\!\left(\op{perimeter}(\mathcal{C}_{x_1})\right),
\]
and therefore,
\[
\#\mathcal{N}(N,L_1',L_1,L_2',L_2)
=
\sum_{x_1\ge 1} \op{area}(\mathcal{C}_{x_1})
+
O\!\left(\sum_{x_1\ge 1}\op{perimeter}(\mathcal{C}_{x_1})\right).
\]
By construction,
\[
V(N,L_1',L_1,L_2',L_2)
=
\int_{0}^{\infty} \op{area}(\mathcal{C}_{x_1})\,dx_1,
\]
since the volume computation in the previous lemma corresponds to
integrating first in $x_1$ and then in $(u,v)$.
The function $x_1\mapsto \op{area}(\mathcal{C}_{x_1})$ is nonnegative,
supported on $x_1\ll N^{1/10}$, and varies smoothly with $x_1$.
Therefore, by a standard comparison between sums and integrals,
\[
\sum_{x_1\ge 1} \op{area}(\mathcal{C}_{x_1})
=
\int_{0}^{\infty} \op{area}(\mathcal{C}_{x_1})\,dx_1
+
O\!\left(\sup_{x_1}\op{area}(\mathcal{C}_{x_1})\right).
\]
Since $\op{area}(\mathcal{C}_{x_1})\ll 1$ uniformly in $x_1$,
this error term is $O(1)$.

It remains to bound the boundary contribution.
The perimeter of $\mathcal{C}_{x_1}$ is $\ll 1$ uniformly for all
admissible $x_1$, and $\mathcal{C}_{x_1}$ is nonempty only for
$x_1\ll N^{1/10}$. Hence
\[
\sum_{x_1\ge 1}\op{perimeter}(\mathcal{C}_{x_1})
\ll N^{1/10}.
\]
Combining these estimates yields
\[
\#\mathcal{N}(N,L_1',L_1,L_2',L_2)
=
V(N,L_1',L_1,L_2',L_2)
+
O\!\left(N^{1/10}\right),
\]
as claimed.
\end{proof}

\begin{lemma}
For $N>0$ let $R_i',R_i$ be real numbers with
$1\le R_1'<R_1$, $0\le R_2'<R_2$, and $0\le R_3'<R_3$ and set $\mathscr{R}:=\prod_{i=1}^3[R_i',R_i]$.
Then
\[
\begin{split}& \#\mathcal{C}(N,\mathscr{R})\\
=&
\frac{75}{8}\,N^{1/5} \left(R_1^{1/15}
- {R_1'}^{1/15}\right)\left(R_2'^{-2/15}
- {R_2}^{-2/15}\right)
\sum_{\substack{a_2 a_4\in [R_3',R_3]}}
\frac{1}{a_2^{3/5} a_4}
+
O\!\left(N^{1/10+\varepsilon}\right),
\end{split}
\]
the implied constant depends at most on $\varepsilon$ and the fixed
interval parameters $R_i'$.
\end{lemma}

\begin{proof}
Fix integers $(a_2,a_4)$ with $a_2 a_4\in [R_3',R_3]$.
For such a pair, the defining conditions reduce to constraints on
$(a_1,a_3,a_5)$ given by
\[
a_1^5 a_3^3 a_5^5 \le \frac{N}{a_2^4 a_4^4},
\qquad a_5\ge a_1,
\]
together with
\[
\frac{a_5^2}{a_1^2}\in
\left[\frac{R_1' a_2}{a_4},\,\frac{R_1 a_2}{a_4}\right],
\qquad
\frac{a_5}{a_1 a_3^3}\in
\left[\frac{R_2' a_4}{a_2},\,\frac{R_2 a_4}{a_2}\right].
\]
Thus, for fixed $(a_2,a_4)$, the triple $(a_1,a_3,a_5)$ ranges over a region
of exactly the same shape as $\mathcal{M}(N,L_1',L_1,L_2',L_2)$ from the
previous lemma, but with the substitutions
\[
N \mapsto \frac{N}{a_2^4 a_4^4},\qquad
L_1' \mapsto \sqrt{\frac{R_1' a_2}{a_4}},\quad
L_1 \mapsto \sqrt{\frac{R_1 a_2}{a_4}},\qquad
L_2' \mapsto \frac{R_2' a_4}{a_2},\quad
L_2 \mapsto \frac{R_2 a_4}{a_2}.
\]
Applying the estimate from Lemma \ref{hashtag N lemma}, we find that the number of choices of $(a_1, a_3, a_5)$ for a fixed pair $(a_2,a_4)$ equals 
\begin{align*}
&\frac{75}{8}
\left(\frac{N}{a_2^4 a_4^4}\right)^{1/5}
\left(
\left(\frac{R_1 a_2}{a_4}\right)^{1/15}
-
\left(\frac{R_1' a_2}{a_4}\right)^{1/15}
\right)
\left(
\left(\frac{R_2' a_4}{a_2}\right)^{-2/15}
-
\left(\frac{R_2 a_4}{a_2}\right)^{-2/15}
\right) \\
&\qquad
+ O\!\left(\left(\frac{N}{a_2^4 a_4^4}\right)^{1/10}\right).
\end{align*}
Since $(a_2^4 a_4^4)^{1/5}=(a_2 a_4)^{4/5}$, the main term contributes
\[
\frac{75}{8}\,N^{1/5}
\frac{1}{a_2^{3/5} a_4}\left(R_1^{1/15}-{R_1'}^{1/15}\right)\left(R_2'^{2/15}-{R_2}^{2/15}\right).
\]
Summing over all $(a_2,a_4)$ with $a_2 a_4\in [R_3',R_3]$ yields the stated
main term. The total error is bounded by
\[
\sum_{a_2 a_4\in [R_3',R_3]}
\left(\frac{N}{a_2^4 a_4^4}\right)^{1/10}
\ll
N^{1/10}
\sum_{n\le R_3} \frac{d(n)}{n^{2/5}}
\ll
N^{1/10+\varepsilon},
\]
where $d(n)$ is the number of divisors of $n$ and the implied constant depends on $R_3$.
\end{proof}

\par Let $\ell$ be a prime number. Then $a=(a_1, \dots, a_5)$ is said to be $\ell$-carefree if $\ell^2\nmid a_i a_j$ for all $i<j$. If $\ell\geq 5$, denote by $\mathcal{C}_\ell(N,\mathscr{R})$ the set of $\ell$-carefree tuples in $\mathcal{C}(N,\mathscr{R})$. At $\ell=2$ or $3$, we impose additional conditions so that $m$ has Type $(i,j)$. Let $\Omega_2\subset \left(\Z/2^{6}\right)^5$ (resp. $\Omega_3\subset \left(\Z/3^{5}\right)^5$) consist of the congruence conditions defining $(Ai)$ (resp. $(Bj)$). For instance, if $(Ai)=(A5)$, then $\Omega_2$ consists of all tuples $(\bar{a}_1, \dots, \bar{a}_5)\in \left(\Z/2^6\right)^5$ such that $4\nmid \bar{a}_i\bar{a}_j$ for all $i<j$ and \[m:=\epsilon \prod_{i=1}^5 a_i^i=48\pmod{64}.\] We set $\Omega_6:=\Omega_2\times \Omega_3\subset \left(\Z/2^63^5\right)^5$. Of the other hand, if $\ell\geq 5$, let $\Omega_\ell\subset \left(\Z/\ell^2\right)^5$ consist of $(\bar{a}_1,\dots, , \bar{a}_5)$ such that $\ell^2\nmid \bar{a}_i\bar{a}_j$ for all $i<j$. Given a squarefree natural number $n$ divisible by $6$, set
\[\Omega_n:=\prod_{\ell|n}\Omega_\ell\subset \left(\Z/N(n)\right)^5\] where $N(n):=2^{6}3^{5}\prod_{\substack{\ell|n\\
\ell\neq 2,3}} \ell^2$. Let $\mathcal{C}_\ell(N,\mathscr{R})$ consist of $(a_1, \dots, a_5)\in \Z^5$ such that 
\[(a_1, \dots, a_5)\pmod{N(\ell)}\in \Omega_\ell.\] Given a squarefree natural number $n$ which is divisible by $6$, set
\[\mathcal{C}^n(N,\mathscr{R}):=\bigcap_{\ell|n} \mathcal{C}_\ell(N,\mathscr{R}).\] We then set 
\[\mathcal{C}_{\op{cf}}(N,\mathscr{R}):=\bigcap_\ell \mathcal{C}_\ell(N,\mathscr{R}),\]where $\ell$ ranges over all prime numbers. Note that the set $\mathcal{C}_{\op{cf}}(N,\mathscr{R})$ can be identified with the set of pure sextic fields $K=\Q(\sqrt[6]{m})$ for which:
\begin{itemize}
    \item $m$ has sign $\epsilon$, 
    \item $m$ is of Type $(i,j)$, i.e., $K$ satisfies the conditions $(Ai)$ and $(Bj)$,
    \item $|\Delta_K|\leq N$.
\end{itemize}
Our goal for the rest of this section is therefore to prove an explicit formula for the limit
\[\lim_{N\rightarrow \infty} \frac{\#\mathcal{C}_{\op{cf}}(N, \mathscr{R})}{N^{1/5}}.\]
\begin{lemma}
  Let $n$ be a squarefree natural number which is divisible by $6$. With respect to notation above, we have that:\[
\#\Omega_n
=
\prod_{\ell\mid n} \#\Omega_\ell
=
(n/6)^{10}\#\Omega_6\prod_{\substack{\ell\mid n\\\ell\neq 2,3}}
\Bigl(1-\frac{1}{\ell}\Bigr)^4
\Bigl(1+\frac{4}{\ell}\Bigr).
\]
\end{lemma}
\begin{proof}It suffices to compute the cardinality of $\Omega_\ell$ for a prime~$\ell\neq 2,3$.
By definition, $\Omega_\ell \subset (\mathbb{Z}/\ell^2\mathbb{Z})^5$
consists of those residue classes $(\bar a_1,\dots,\bar a_5)$ such that
for every pair $i<j$ one has $\ell^2 \nmid \bar a_i \bar a_j$. Thus, a residue class lies in $\Omega_\ell$ if and only if
at most one of the coordinates $\bar a_i$ is divisible by~$\ell$. If none of the coordinates is divisible by~$\ell$, then each
$\bar a_i$ may be chosen in $(\mathbb{Z}/\ell^2\mathbb{Z})^\times$,
giving $(\ell^2-\ell)^5$ possibilities.

If exactly one coordinate, say $\bar a_i$, is divisible by~$\ell$,
then $\bar a_i$ may be chosen in the $\ell$ residue classes
$\ell\mathbb{Z}/\ell^2\mathbb{Z}$, while the remaining four coordinates
must be units modulo~$\ell$, yielding
$\ell(\ell^2-\ell)^4$ possibilities for each choice of~$i$.
There are $5$ choices for the index~$i$. Therefore,
\[
\#\Omega_\ell
=
(\ell^2-\ell)^5
+
5\,\ell(\ell^2-\ell)^4
=
\ell^{10}\Bigl(1-\frac{1}{\ell}\Bigr)^4
\Bigl(1+\frac{4}{\ell}\Bigr).
\]
\noindent Since $\Omega_n=\prod_{\ell\mid n}\Omega_\ell$, the result follows.
\end{proof} We now use this local computation to estimate
$\#\mathcal{C}^n(N, \mathscr{R})$.
By construction, $\mathcal{C}^n(N, \mathscr{R})$ consists of those
$a\in\mathcal{C}(N, \mathscr{R})$ whose reduction modulo $\ell^2$ lies in
$\Omega_\ell$ for every prime $\ell\mid n$ with additional conditions at $\ell=2,3$.

The congruence conditions modulo $N(n)$ define a union of
$\#\Omega_n$ residue classes in $(\mathbb{Z}/N(n)\mathbb{Z})^5$. Let $n_{i,j}(a_2, a_4)$ be the number of tuples $(\bar{a}_1, \bar{a}_3, \bar{a}_5)\in (\Z/2^{6}3^{5})^3$ such that
\begin{itemize}
    \item $2^2\nmid \bar{a}_1\bar{a}_3, \bar{a}_1\bar{a}_5, \bar{a}_3\bar{a}_5$,
    \item $3^2\nmid \bar{a}_1\bar{a}_3, \bar{a}_1\bar{a}_5, \bar{a}_3\bar{a}_5$,
    \item $(\bar{a}_1, \bar{a}_2, \dots, \bar{a}_5)$ belongs to $\Omega_6$.
\end{itemize}

\begin{lemma}\label{lem:Rn-asymp}
Let $n$ be a squarefree positive integer divisible by $6$ and all primes $\ell\leq R_3$. Then one has
\[
\begin{split}\#\mathcal{C}^n(N,\mathscr{R})
=&
\frac{25 N^{1/5}}{124416}\prod_{\substack{\ell|n\\\ell\neq 2,3}}\left(1-\frac{3}{\ell^2}+\frac{2}{\ell^3}\right)
\bigl(R_1^{1/15}-R_1'^{1/15}\bigr)
\bigl(R_2'^{-2/15}-R_2^{-2/15}\bigr) \\
\times &\sum_{\substack{a_2a_4\in [R_3', R_3]\\
    a_2a_4\text{ is squarefree}}}\frac{n_{i,j}(a_2,a_4)\prod_{\substack{\ell|a_2a_4\\ \ell\neq 2,3}}\left(\frac{\ell-1}{\ell+2}\right)}{a_2^{3/5} a_4}+
O\!\left(N^{1/10}\right),
\end{split}\]
where the implied constant depends on $\varepsilon$, $n$, and the fixed
parameters $R_i,R_i'$.
\end{lemma}

\begin{proof}
Let $\Lambda_n=(N(n)\mathbb{Z})^5\subset\mathbb{Z}^5$. Since $n$ is squarefree,
the condition that a tuple $a=(a_1,\dots,a_5)$ be $\ell$-carefree for every
prime $\ell\mid n$ depends only on the reduction of $a$ modulo $\ell^2$, and
hence only on its reduction modulo $N(n)$. Consequently, the set
$\Omega_n\subset(\mathbb{Z}/N(n)\mathbb{Z})^5$ parametrizes exactly the residue
classes modulo $\Lambda_n$ that are allowed by the local carefree conditions such that $m$ is of Type $(i,j)$.
Fix once and for all a complete set of representatives
\[
\widetilde{\Omega}_n\subset ([0,N(n))\cap\mathbb{Z})^5
\]
for $\Omega_n$.

Every element of $\mathcal{C}^n(N,\mathscr{R})$ lies in a unique coset of $\Lambda_n$
corresponding to some $z\in\widetilde{\Omega}_n$, and therefore
\[
\mathcal{C}^n(N,\mathscr{R})
=
\bigsqcup_{z\in\widetilde{\Omega}_n}
\bigl((z+\Lambda_n)\cap\mathcal{C}(N,\mathscr{R})\bigr).
\]
Taking cardinalities and translating by $-z$, we obtain
\[
\#\mathcal{C}^n(N,\mathscr{R})
=
\sum_{z\in\widetilde{\Omega}_n}
\#\Bigl(\Lambda_n\cap(\mathcal{C}(N,\mathscr{R})-z)\Bigr)
=
\sum_{z\in\widetilde{\Omega}_n}
\#\Bigl(\mathbb{Z}^5\cap\bigl(\tfrac{1}{N(n)}\mathcal{C}(N,\mathscr{R})-\tfrac{z}{N(n)}\bigr)\Bigr).
\]

Fix $z=(z_1,\dots,z_5)\in\widetilde{\Omega}_n$. Writing
$a_i=z_i+N(n) b_i$ for $i=1,\dots,5$, the defining conditions of
$\mathcal{C}(N,\mathscr{R})$ translate into conditions on the integer vector
$b=(b_1,\dots,b_5)$ describing the shifted and rescaled region
\[
\tfrac{1}{N(n)}\mathcal{C}(N,\mathscr{R})-\tfrac{z}{N(n)}\subset\mathbb{R}^5.
\]
We now fix integers $b_2,b_4$ such that
\[
a_2=z_2+N(n) b_2>0,\qquad a_4=z_4+N(n) b_4>0,
\qquad\text{and}\qquad
a_2 a_4\in [R_3',R_3].
\]
For such a fixed pair $(a_2,a_4)$, we find that $(b_1,b_3,b_5)\in\mathbb{Z}^3$ must lie in the translated region
\[
\mathcal{M}\!\left(
\frac{N}{n^{26}a_2^4 a_4^4},
L_1',L_1,L_2',L_2
\right)
-
\frac{1}{N(n)}(z_1,z_3,z_5),
\]
where the shape parameters are given by
\[
L_1'=\sqrt{\frac{R_1' a_2}{a_4}},\qquad
L_1=\sqrt{\frac{R_1 a_2}{a_4}},\qquad
L_2'=n^6\frac{R_2' a_4}{a_2},\qquad
L_2=n^6\frac{R_2 a_4}{a_2}.
\]
We find that 
\[\begin{split}
&\#\Bigl(\mathbb{Z}^5\cap\bigl(\tfrac{1}{N(n)}\mathcal{C}(N,\mathscr{R})-\tfrac{z}{N(n)}\bigr)\Bigr)\\
=&\sum_{\substack{b_2,b_4\geq 0\\
(z_2+N(n)b_2)(z_4+N(n)b_4)\in [R_3', R_3]}}\# \left(\Z^3\cap\left(\mathcal{M}\!\left(
\frac{N}{n^{26}(z_2+N(n)b_2)^4 (z_4+N(n)b_4)^4},
L_1',L_1,L_2',L_2
\right)
-
\frac{1}{N(n)}(z_1,z_3,z_5)\right)\right),
\end{split}\]
where the values of $L_i$ and $L_i'$ are given above. We have that 
\[\begin{split}&\# \left(\Z^3\cap\left(\mathcal{M}\!\left(
\frac{N}{n^{26}(z_2+N(n)b_2)^4 (z_4+N(n)b_4)^4},
L_1',L_1,L_2',L_2
\right)
-
\frac{1}{N(n)}(z_1,z_3,z_5)\right)\right)\\
=& V\left(\mathcal{M}\!\left(
\frac{N}{n^{26}(z_2+N(n)b_2)^4 (z_4+N(n)b_4)^4},
L_1',L_1,L_2',L_2
\right)\right)+O(N^{\frac{1}{10}})\\
=& \frac{75}{8}\,\frac{N^{1/5}}{n^{\frac{26}{5}}(z_2+N(n)b_2)^{\frac{4}{5}} (z_4+N(n)b_4)^{\frac{4}{5}}}
\bigl(L_1^{2/15}-L_1'^{2/15}\bigr)
\bigl(L_2'^{-2/15}-L_2^{-2/15}\bigr)+O\left( N^{\frac{1}{10}}\right)\\
=&\frac{75}{8}\,\frac{N^{1/5}}{n^{6}(z_2+N(n)b_2)^{3/5} (z_4+N(n)b_4)}
\bigl(R_1^{1/15}-R_1'^{1/15}\bigr)
\bigl(R_2'^{-2/15}-R_2^{-2/15}\bigr)+O\left( N^{\frac{1}{10}}\right).
\end{split}\]
We deduce that 
\[\begin{split}&\#\mathcal{C}^n(N,\mathscr{R})\\
=&\frac{75}{8}\,\frac{N^{1/5}}{n^{6}}
\bigl(R_1^{1/15}-R_1'^{1/15}\bigr)
\bigl(R_2'^{-2/15}-R_2^{-2/15}\bigr) \sum_{z\in \widetilde{\Omega}_n}\sum_{\substack{b_2,b_4\geq 0\\
(z_2+N(n)b_2)(z_4+N(n)b_4)\in [R_3', R_3]}}\frac{1}{(z_2+N(n)b_2)^{3/5} (z_4+N(n)b_4)}\\
=&\frac{75}{8}\,\frac{N^{1/5}}{n^{6}}
\bigl(R_1^{1/15}-R_1'^{1/15}\bigr)
\bigl(R_2'^{-2/15}-R_2^{-2/15}\bigr) \sum_{a_2a_4\in [R_3', R_3]}\frac{\#\{z\in \widetilde{\Omega}_n\mid z_i\equiv a_i\pmod{N(n)}\text{ for }i=2,4\}}{a_2^{3/5} a_4}.
\end{split}\]
From the decomposition $\widetilde{\Omega}_n=\prod_{\ell|n} \widetilde{\Omega}_\ell$ one finds that 
\[\#\{z\in \widetilde{\Omega}_n\mid z_i\equiv a_i\pmod{N(n)}\text{ for }i=2,4\}=\prod_{\ell|n} \#\{z\in \widetilde{\Omega}_\ell\mid z_i\equiv a_i\pmod{\ell^2}\text{ for }i=2,4\}.\]
\par Let's count the number of $z\in \widetilde{\Omega}_\ell$ for which $z_2$ and $z_4$ are fixed. Assume that $\ell\neq 2,3$. If $z_2$ or $z_4$ is $0$, or if $\ell$ divides both $z_2$ and $z_4$ then there are no choices of $z$. Consider the case when $\ell\nmid z_2z_4$. In this case the number of choices for $(z_1, z_3, z_5)$ is $\ell^{6}\left(1-\frac{1}{\ell}\right)^2\left(1+\frac{2}{\ell}\right)$. Finally, consider the case when $\ell|z_2z_4$, but $z_2z_4\neq 0$. In this case none of the coordinates $(z_1, z_3, z_5)$ are divisible by $\ell$ and thus the number of choices in this case is $(\ell^2-\ell)^3=\ell^6(1-\frac{1}{\ell})^3$. Thus we have the following cases:
\[\#\{z\in \widetilde{\Omega}_\ell\mid z_i\equiv a_i\pmod{\ell^2}\text{ for }i=2,4\}=\begin{cases}
    0&\text{ if }\ell^2|a_2a_4,\\
    \ell^{6}\left(1-\frac{1}{\ell}\right)^2\left(1+\frac{2}{\ell}\right)&\text{ if }\ell\nmid a_2a_4,\\
    \ell^6(1-\frac{1}{\ell})^3&\text{ if }\ell|a_2a_4\text{ and }\ell^2\nmid a_2a_4.
\end{cases}\]
Note that by assumption on $n$, all primes that divide $a_2a_4$ must also divide $n$. Thus we deduce that 
\[\begin{split}
    & \sum_{a_2a_4\in [R_3', R_3]}\frac{\#\{z\in \widetilde{\Omega}_n\mid z_i\equiv a_i\pmod{N(n)}\text{ for }i=2,4\}}{a_2^{3/5} a_4}\\
    =& \sum_{\substack{a_2a_4\in [R_3', R_3]\\
    a_2a_4\text{ is squarefree}}}\frac{n_{i,j}(a_2,a_4)\prod_{\substack{\ell|n, \ell\nmid a_2a_4\\ \ell\neq 2,3}}\ell^{6}\left(1-\frac{1}{\ell}\right)^2\left(1+\frac{2}{\ell}\right)\times \prod_{\substack{ \ell\mid a_2a_4\\ \ell\neq 2,3}}\ell^6(1-\frac{1}{\ell})^3}{a_2^{3/5} a_4}\\
    =& (n/6)^6\prod_{\substack{\ell|n\\\ell\neq 2,3}}\left(1-\frac{3}{\ell^2}+\frac{2}{\ell^3}\right)\sum_{\substack{a_2a_4\in [R_3', R_3]\\
    a_2a_4\text{ is squarefree}}}\frac{n_{i,j} (a_2,a_4)\prod_{\substack{\ell|a_2a_4\\\ell\neq 2,3}}\left(\frac{\ell-1}{\ell+2}\right)}{a_2^{3/5} a_4}
\end{split}.\]
Thus we have shown that 
\[\begin{split}\#\mathcal{C}^n(N,\mathscr{R})=&\frac{25 N^{1/5}}{124416}\prod_{\substack{\ell|n\\\ell\neq 2,3}}\left(1-\frac{3}{\ell^2}+\frac{2}{\ell^3}\right)
\bigl(R_1^{1/15}-R_1'^{1/15}\bigr)
\bigl(R_2'^{-2/15}-R_2^{-2/15}\bigr) \\
\times& \sum_{\substack{a_2a_4\in [R_3', R_3]\\
    a_2a_4\text{ is squarefree}}}\frac{n_{i,j}(a_2,a_4)\prod_{\substack{\ell|a_2a_4\\ \ell\neq 2,3}}\left(\frac{\ell-1}{\ell+2}\right)}{a_2^{3/5} a_4}+
O\!\left(N^{1/10}\right).\end{split}\]
\end{proof}
Given a prime number $\ell$, let $\mathcal{C}_\ell'(N,\mathscr{R})$ be the complement of $\mathcal{C}_\ell(N,\mathscr{R})$. 

\begin{lemma}\label{Rellprime lemma} With respect to notation above, we have that
    \[\frac{\#\mathcal{C}_\ell'(N,\mathscr{R})}{N^{1/5}}=O\left(\ell^{-6/5}+N^{-1/10}\right).\]
\end{lemma}
\begin{proof}
Since $\frac{\#\mathcal{C}(N,\mathscr{R})}{N^{1/5}}=O(1)$, it suffices to prove the result for sufficiently large values of $\ell$. Assume therefore without loss of generality that $\ell\geq 5$. Let $a=(a_1, \dots, a_5)\in \mathcal{C}_\ell'(N,\mathscr{R})$, then at least one of the following cases occur:
\begin{itemize}
    \item $\ell^2|a_i$ for some $1\leq i\leq 5$,
    \item $\ell^2|a_ia_j$ for some $1\leq i<j\leq 5$.
\end{itemize}
Suppose for instance that $\ell^2|a_3$. Then, $a=(a_1,a_2,\ell^2a_3', a_4, a_5)$ where $(a_1, a_2,a_3', a_4, a_5)\in \mathcal{C}(N/\ell^6)$ with some modified shape parameters. Thus, the number of choices of $(a_1, a_2,a_3, a_4, a_5)$ for which $\ell^2|a_3$ is $O\left(\frac{N^{1/5}}{\ell^{6/5}}+{N^{1/10}}\right)$. Similar arguments show that in all cases, the number of $a$ for which $\ell^2|a_i$ (resp. $\ell^2|a_ia_j$) is $O\left(\frac{N^{1/5}}{\ell^{m/5}}+{N^{1/10}}\right)$ where $m\in \{6,7,8,9,10\}$. Combining these estimates, we obtain the result.
\end{proof}
\noindent Given a positive real number $Y$, let $\mathcal{D}_Y(N,\mathscr{R}):=\bigcup_{\ell>Y} \mathcal{C}_\ell'(N,\mathscr{R})$. We think of the above as the tail for our sieve estimate. \begin{lemma}\label{tail estimate}
    With respect to notation above, 
    \[\frac{\#\mathcal{D}_Y(N,\mathscr{R}) }{N^{1/5}}=O\left(\sum_{\ell>Y} \ell^{-6/5}+\frac{1}{\log N}\right).\]
\end{lemma}
\begin{proof}
    Suppose $a=(a_1, \dots, a_5)$ belongs to $\mathcal{C}_\ell'(N,\mathscr{R})$. Thus we find that $\ell^2$ divides $a_ia_j$ for some $1\leq i, j\leq 5$. If $i$ or $j$ is $2$ or $4$, then $\ell$ must be bounded if $a_2a_4\in [R_3', R_3]$. Thus in this case $\ell$ is bounded. Suppose that $\ell^2\nmid a_1a_5$ and $\ell|a_3$. Then the condition $\lambda_2(m)^3\in [R_2',R_2]$ shall imply that $\ell$ is bounded. Finally consider the case when $\ell^2|a_1a_5$. Then since $a_1^5a_2^4a_3^3a_4^4a_5^5\leq N$, it follows that $\ell\ll N^{1/10}$. Thus we have shown that  $\mathcal{C}_\ell'(N,\mathscr{R})=\emptyset$ unless $\ell\ll N^{1/10}$. Thus from Lemma \ref{Rellprime lemma},
     \[\frac{\# \mathcal{D}_Y(N,\mathscr{R})}{N}=O\left(\sum_{\ell>Y} \ell^{-6/5}+ \frac{\pi(N^{1/10})}{N^{1/10}}\right),\]
    and the result follows from the Prime number theorem.
\end{proof}
\begin{proposition}\label{main propn} We have the following limit:
\[
\begin{split}&\lim_{N\rightarrow\infty}\frac{\#\mathcal{C}_{\op{cf}}(N,\mathscr{R})}{N^{1/5}}\\
=&
\frac{25}{124416}\prod_{\substack{\ell\neq 2,3}}\left(1-\frac{3}{\ell^2}+\frac{2}{\ell^3}\right)
\bigl(R_1^{1/15}-R_1'^{1/15}\bigr)
\bigl(R_2'^{-2/15}-R_2^{-2/15}\bigr) \sum_{\substack{a_2a_4\in [R_3', R_3]\\
    a_2a_4\text{ is squarefree}}}\frac{n_{i,j}(a_2,a_4)\prod_{\substack{\ell|a_2a_4\\ \ell\neq 2,3}}\left(\frac{\ell-1}{\ell+2}\right)}{a_2^{3/5} a_4}.
    \end{split}
\]    
\end{proposition}

\begin{proof}
Let $n$ be a squarefree positive integer divisible by $6$ and all primes $\ell\leq R_3$. From Lemma \ref{lem:Rn-asymp}, one has
\[
\begin{split}&\lim_{N\rightarrow \infty}\frac{\#\mathcal{C}^n(N,\mathscr{R})}{N^{1/5}}\\
=&
\frac{25}{124416}\prod_{\substack{\ell|n\\\ell\neq 2,3}}\left(1-\frac{3}{\ell^2}+\frac{2}{\ell^3}\right)
\bigl(R_1^{1/15}-R_1'^{1/15}\bigr)
\bigl(R_2'^{-2/15}-R_2^{-2/15}\bigr) \sum_{\substack{a_2a_4\in [R_3', R_3]\\
    a_2a_4\text{ is squarefree}}}\frac{n_{i,j}(a_2,a_4)\prod_{\substack{\ell|a_2a_4\\ \ell\neq 2,3}}\left(\frac{\ell-1}{\ell+2}\right)}{a_2^{3/5} a_4}.
    \end{split}
\]
For $Y>0$, set $n_Y:=\prod_{\ell\leq Y} \ell$. Since
\[\mathcal{C}_{\op{cf}}(N,\mathscr{R})\subseteq \mathcal{C}_{n_Y}(N,\mathscr{R}),\]we find that
\[\limsup_{N\rightarrow \infty} \frac{\# \mathcal{C}_{\op{cf}}(N,\mathscr{R})}{N^{1/5}}\leq \limsup_{N\rightarrow \infty} \frac{\# \mathcal{C}_{n_Y}(N,\mathscr{R})}{N^{1/5}}=\lim_{N\rightarrow \infty} \frac{\# \mathcal{C}_{n_Y}(N,\mathscr{R})}{N^{1/5}}.\]
Therefore, we deduce that
\[\begin{split}&\limsup_{N\rightarrow \infty} \frac{\# \mathcal{C}_{\op{cf}}(N,\mathscr{R})}{N^{1/5}}\\
\leq &\frac{25}{124416}\prod_{\substack{\ell\leq Y\\\ell\neq 2,3}}\left(1-\frac{3}{\ell^2}+\frac{2}{\ell^3}\right)
\bigl(R_1^{1/15}-R_1'^{1/15}\bigr)
\bigl(R_2'^{-2/15}-R_2^{-2/15}\bigr) \sum_{\substack{a_2a_4\in [R_3', R_3]\\
    a_2a_4\text{ is squarefree}}}\frac{n_{i,j}(a_2,a_4)\prod_{\substack{\ell|a_2a_4\\ \ell\neq 2,3}}\left(\frac{\ell-1}{\ell+2}\right)}{a_2^{3/5} a_4}.\end{split}\]
    Taking the limit as $Y\rightarrow \infty$, we find that
    \[\begin{split}&\limsup_{N\rightarrow \infty} \frac{\# \mathcal{C}_{\op{cf}}(N,\mathscr{R})}{N^{1/5}}\\
    \leq& \frac{25}{124416}\prod_{\substack{\ell\neq 2,3}}\left(1-\frac{3}{\ell^2}+\frac{2}{\ell^3}\right)
\bigl(R_1^{1/15}-R_1'^{1/15}\bigr)
\bigl(R_2'^{-2/15}-R_2^{-2/15}\bigr) \sum_{\substack{a_2a_4\in [R_3', R_3]\\
    a_2a_4\text{ is squarefree}}}\frac{n_{i,j}(a_2,a_4)\prod_{\substack{\ell|a_2a_4\\ \ell\neq 2,3}}\left(\frac{\ell-1}{\ell+2}\right)}{a_2^{3/5} a_4}.\end{split}\]
    In order to complete the proof, it suffices to show that $\liminf_{N\rightarrow \infty} \frac{\# \mathcal{C}_{\op{cf}}(N,\mathscr{R})}{N^{1/5}}$ is at least the right hand side of the above. Since 
    \[\mathcal{C}_{n_Y}(N,\mathscr{R})\subseteq \mathcal{C}_{\op{cf}}(N,\mathscr{R})\cup \mathcal{D}_Y(N,\mathscr{R}),\] it follows that 
    \[\liminf_{N\rightarrow \infty} \frac{\# \mathcal{C}_{\op{cf}}(N,\mathscr{R})}{N^{1/5}}\geq \lim_{N\rightarrow\infty} \frac{\# \mathcal{C}_{n_Y}(N,\mathscr{R})}{N^{1/5}}-\limsup_{N\rightarrow\infty} \frac{\# \mathcal{D}_Y(N,\mathscr{R})}{N^{1/5}}.\]
    It follows from Lemma \ref{tail estimate} that
    \[\limsup_{N\rightarrow\infty} \frac{\# \mathcal{D}_Y(N,\mathscr{R})}{N^{1/5}}\leq C \sum_{\ell>Y} \ell^{-6/5}\] for some positive constant $C>0$. Since the limit of $\sum_{\ell>Y} \ell^{-6/5}$ goes to $0$ as $Y\rightarrow \infty$, it follows that 
    \[\begin{split}&\liminf_{N\rightarrow \infty} \frac{\# \mathcal{C}_{\op{cf}}(N,\mathscr{R})}{N^{1/5}}\\
    \geq&\lim_{Y\rightarrow \infty} \lim_{N\rightarrow\infty} \frac{\# \mathcal{C}_{n_Y}(N,\mathscr{R})}{N^{1/5}}\\
    =& \frac{25}{124416}\prod_{\substack{\ell\neq 2,3}}\left(1-\frac{3}{\ell^2}+\frac{2}{\ell^3}\right)
\bigl(R_1^{1/15}-R_1'^{1/15}\bigr)
\bigl(R_2'^{-2/15}-R_2^{-2/15}\bigr) \sum_{\substack{a_2a_4\in [R_3', R_3]\\
    a_2a_4\text{ is squarefree}}}\frac{n_{i,j}(a_2,a_4)\prod_{\substack{\ell|a_2a_4\\ \ell\neq 2,3}}\left(\frac{\ell-1}{\ell+2}\right)}{a_2^{3/5} a_4}.\end{split}\]
\end{proof}
We now state and prove the main result of this section. Let $\mathscr{R}:=\prod_{i=1}^3 [R_i', R_i]$ and \[\alpha(n):=\prod_{\substack{\ell|n\\\ell\neq 2,3}} \frac{\ell-1}{\ell+2} \sum_{n_1 n_2=n} \frac{n_{i,j}(n_1, n_2)}{n_1^{3/5}n_2}\] for $n$ squarefree and $0$ otherwise. Given a natural number $n$, let $\delta_n$ be the associated Dirac measure supported at $n$. Let $\mu_\alpha$ be the discrete measure supported at squarefree natural numbers defined by:
\[\mu_\alpha(x)dx:=\sum_{\substack{n\geq 1\\ n\text{ is squarefree}}}\alpha(x)\delta_n(x)dx.\]
Denote by $\hat{\mu}_{i,j}$ the product 
\begin{equation}\label{mu i,j defn}\hat{\mu}_{i,j}:=\frac{\prod_{\substack{\ell\neq 2,3}}\left(1-\frac{3}{\ell^2}+\frac{2}{\ell^3}\right)}{559872}x_1^{-14/15}x_2^{-17/15}\mu_\alpha(x_3) dx_1dx_2dx_3.\end{equation}
\begin{theorem}\label{main thm}
    With respect to notation above, consider the set of pure sextic fields $K(\sqrt{m})$ with $m$ of sign $\epsilon$ and Type $(i, j)$ (i.e., satisfying $(Ai, Bj)$) ordered by their absolute discriminant. Then we have that 
    \[\lim_{N\rightarrow \infty} \frac{\#\{K\mid m\text{ is of sign }\epsilon \text{ of Type }(i,j)\text{ with }(\lambda_1(m)^3,\, \lambda_2(m)^3,\, \lambda_3(m)^{-1})\in \mathscr{R}\text{, and }|\Delta_K|\leq 2^{v_1}3^{v_2}N\}}{N^{1/5}}=\int_{\mathscr{R}} \hat{\mu}_{i,j}.\]
\end{theorem}

\begin{proof}
    The LHS of the above relation coincides with the limit 
    \[\lim_{N\rightarrow \infty} \frac{\# \mathcal{C}_{\op{cf}}(N,\mathscr{R})}{N^{1/5}}\] and it is easy to see that the RHS equals 
    \[\frac{25}{124416}\prod_{\substack{\ell\neq 2,3}}\left(1-\frac{3}{\ell^2}+\frac{2}{\ell^3}\right)
\bigl(R_1^{1/15}-R_1'^{1/15}\bigr)
\bigl(R_2'^{-2/15}-R_2^{-2/15}\bigr) \sum_{\substack{a_2a_4\in [R_3', R_3]\\
    a_2a_4\text{ is squarefree}}}\frac{n_{i,j}(a_2,a_4)\prod_{\substack{\ell|a_2a_4\\ \ell\neq 2,3}}\left(\frac{\ell-1}{\ell+2}\right)}{a_2^{3/5} a_4}.\] The result thus follows from Proposition \ref{main propn}.
\end{proof}

\section{Distribution of the parameters $(a_5/a_1,\, a_2a_4,\, a_3)$}\label{s 5}

\par
Throughout this section, we fix a sign $\epsilon\in\{\pm 1\}$ and a local Type
$(i,j)$ determined by congruence conditions at $2$ and $3$.
We consider pure sextic fields
\[
K=\Q(\sqrt[6]{m}), \qquad 
m=\epsilon a_1 a_2^2 a_3^3 a_4^4 a_5^5,
\]
where $a_1,\dots,a_5$ are positive integers satisfying the usual coprimality
conditions ensuring that $m$ is sixth–power-free.
According to \eqref{v_1 v_2}, the absolute discriminant of $K$ is given by
\[
|\Delta_K|=2^{v_1}3^{v_2}a_1^5 a_2^4 a_3^3 a_4^4 a_5^5,
\]
up to a bounded factor depending only on the local Type $(i,j)$.

\par Once the discrete parameters
$a'=(a_2,a_3,a_4)$ are fixed, the remaining shape parameters vary along a
one–parameter family indexed by the ratio $a_5/a_1$.
This motivates the present section, in which we study the joint distribution of
the triple
\[
\left(\frac{a_5}{a_1},\, a_2a_4,\, a_3\right),
\]
ordered by absolute discriminant.

\par Fix a rectangular region
\[
\mathscr{R}=[R_1',R_1]\times [R_2',R_2]\times [R_3',R_3]\subset [1,\infty)^3,
\]
and for $N>0$ define $\mathcal{T}(N,\mathscr{R})$ to be the set of integer tuples
$a=(a_1,\dots,a_5)\in\Z_{>0}^5$ satisfying
\begin{itemize}
    \item $a_1^5 a_2^4 a_3^3 a_4^4 a_5^5 \le N$,
    \item $\left(a_5/a_1,\, a_2a_4,\, a_3\right)\in \mathscr{R}$.
\end{itemize}
Let $\mathcal{T}_{\op{cf}}(N,\mathscr{R})$ denote the subset of
$\mathcal{T}(N,\mathscr{R})$ consisting of tuples satisfying the coprimality
conditions defining sixth–power-free $m$, together with the congruence
conditions defining the fixed Type $(i,j)$.
We begin by counting $\mathcal{T}(N,\mathscr{R})$ without coprimality or local
restrictions. For $0<L_1'\leq L_1<\infty$ and $M>0$, let \[\mathcal{M}(M, L_1', L_1):=\{(x_1, x_5)\in \mathbb{R}^2\mid x_1x_5\leq M,\text{ and }x_5/x_1\in [L_1',L_1]\}\] and let $A(M, L_1', L_1)$ denote the area of $\mathcal{M}(M, L_1', L_1)$.

\begin{lemma}
With respect to notation above,
\[
A(M,L_1',L_1)
=
\frac{M}{2}\log\!\left(\frac{L_1}{L_1'}\right).
\]
\end{lemma}

\begin{proof}
The condition $x_5/x_1\in[L_1',L_1]$ is equivalent to
\[
L_1' x_1 \le x_5 \le L_1 x_1.
\]
For a fixed value of $x_1>0$, the additional constraint $x_1x_5\le M$ gives
\[
x_5\le \frac{M}{x_1}.
\]
Hence, for each $x_1>0$, the admissible values of $x_5$ are those satisfying
\[
L_1' x_1 \le x_5 \le \min\!\left\{L_1 x_1,\frac{M}{x_1}\right\}.
\]
Accordingly, we split the integral into two ranges. When $0<x_1\le (M/L_1)^{1/2}$ we find that $L_1x_1\le M/x_1$, and hence
\[
\min\!\left\{L_1 x_1,\frac{M}{x_1}\right\}=L_1x_1.
\]
The length of the admissible interval in $x_5$ is therefore
\[
(L_1-L_1')x_1.
\]
Thus the contribution to the area from this range is
\[
\int_0^{(M/L_1)^{1/2}} (L_1-L_1')x_1\,dx_1
=
\frac{L_1-L_1'}{2}\left(\frac{M}{L_1}\right).
\]
\noindent Next consider the range $(M/L_1)^{1/2}\le x_1\le (M/L_1')^{1/2}$. In this range, $M/x_1\le L_1x_1$, and hence
\[
\min\!\left\{L_1 x_1,\frac{M}{x_1}\right\}=\frac{M}{x_1}.
\]
The length of the admissible interval in $x_5$ is therefore
\[
\frac{M}{x_1}-L_1'x_1.
\]
The contribution to the area from this range is
\[
\int_{(M/L_1)^{1/2}}^{(M/L_1')^{1/2}}
\left(\frac{M}{x_1}-L_1'x_1\right)dx_1.
\]
Evaluating this integral gives
\[
M\log\!\left(\frac{(M/L_1')^{1/2}}{(M/L_1)^{1/2}}\right)
-
\frac{L_1'}{2}
\left(
\frac{M}{L_1'}-\frac{M}{L_1}
\right)
=
\frac{M}{2}\log\!\left(\frac{L_1}{L_1'}\right)
-
\frac{M}{2}
+
\frac{M L_1'}{2L_1}.
\]

\smallskip
\noindent
Adding the contributions from the two ranges, we obtain
\[
A(M,L_1',L_1)
=
\frac{L_1-L_1'}{2}\cdot\frac{M}{L_1}
+
\frac{M}{2}\log\!\left(\frac{L_1}{L_1'}\right)
-
\frac{M}{2}
+
\frac{M L_1'}{2L_1}.
\]
A straightforward simplification shows that all algebraic terms cancel, leaving
\[
A(M,L_1',L_1)
=
\frac{M}{2}\log\!\left(\frac{L_1}{L_1'}\right),
\]
as claimed.
\end{proof}
Set $\mathcal{N}(M, L_1', L_1):=\mathcal{M}(M, L_1', L_1)\cap \Z^2$, we obtain an asymptotic for $\# \mathcal{N}(M, L_1', L_1)$. 

\begin{lemma}
With respect to the notation above, one has
\[
\#\mathcal{N}(M,L_1',L_1)
=
\frac{M}{2}\log\!\left(\frac{L_1}{L_1'}\right)
+
O\!\left(M^{1/2}\right),
\]
where the implied constant depends only on $L_1'$ and $L_1$.
\end{lemma}

\begin{proof}
By definition, $\mathcal{N}(M,L_1',L_1)$ is the set of integer lattice points in the planar region
\[
\mathcal{M}(M,L_1',L_1)
=
\{(x_1,x_5)\in\mathbb{R}_{>0}^2 \mid x_1x_5\le M,\ x_5/x_1\in[L_1',L_1]\}.
\]
From the previous lemma, the area of this region is
\[
\op{Area}\big(\mathcal{M}(M,L_1',L_1)\big)
=
\frac{M}{2}\log\!\left(\frac{L_1}{L_1'}\right).
\]

To estimate the number of lattice points, we apply the Lipschitz principle.
It suffices to bound the length of the boundary $\partial\mathcal{M}(M,L_1',L_1)$.
This boundary consists of four smooth arcs:
\begin{enumerate}
\item the curve $x_5=L_1x_1$ for $0<x_1\le (M/L_1)^{1/2}$;
\item the curve $x_5=M/x_1$ for $(M/L_1)^{1/2}\le x_1\le (M/L_1')^{1/2}$;
\item the curve $x_5=L_1'x_1$ for $0<x_1\le (M/L_1')^{1/2}$;
\item the vertical segment $x_1=(M/L_1')^{1/2}$.
\end{enumerate}

The lengths of the linear boundary components are $O(M^{1/2})$.
For the hyperbolic arc $x_5=M/x_1$, its length is
\[
\int_{(M/L_1)^{1/2}}^{(M/L_1')^{1/2}}
\sqrt{1+\left(\frac{d}{dx_1}\frac{M}{x_1}\right)^2}\,dx_1
=
\int_{(M/L_1)^{1/2}}^{(M/L_1')^{1/2}}
\sqrt{1+\frac{M^2}{x_1^4}}\,dx_1
\ll
\int_{(M/L_1)^{1/2}}^{(M/L_1')^{1/2}}
\frac{M}{x_1^2}\,dx_1
\ll M^{1/2}.
\]
Thus the total boundary length satisfies
\[
\op{length}\big(\partial\mathcal{M}(M,L_1',L_1)\big)=O(M^{1/2}).
\]
\noindent The Lipschitz principle therefore yields
\[
\#\mathcal{N}(M,L_1',L_1)
=
\op{Area}\big(\mathcal{M}(M,L_1',L_1)\big)
+
O\!\left(\op{length}\big(\partial\mathcal{M}(M,L_1',L_1)\big)\right),
\]
which gives
\[
\#\mathcal{N}(M,L_1',L_1)
=
\frac{M}{2}\log\!\left(\frac{L_1}{L_1'}\right)
+
O\!\left(M^{1/2}\right),
\]
as claimed.
\end{proof}

\par Thus we have that 
\[\begin{split}&\#\mathcal{T}(N,\mathscr{R})\\
=&\sum_{a_2a_4\in [R_2',R_2]}\sum_{a_3\in [R_3',R_3]} \#\mathcal{N}\left(N/a_2^4a_3^3a_4^4, R_1', R_1\right)\\
=& \frac{N}{2}\log\!\left(\frac{R_1}{R_1'}\right)\sum_{a_2a_4\in [R_2',R_2]}\sum_{a_3\in [R_3',R_3]}\frac{1}{a_2^4a_3^3a_4^4}+O(N^{1/2}).
\end{split}\]

We now estimate $\mathcal{T}_{\op{cf}}(N,\mathscr{R})$ by imposing the local
carefree conditions directly on $\mathcal{T}(N,\mathscr{R})$.

For each prime $\ell$, let $\mathcal{T}_\ell(N,\mathscr{R})$ denote the subset
of $\mathcal{T}(N,\mathscr{R})$ consisting of tuples
$a=(a_1,\dots,a_5)$ whose reduction modulo $\ell^2$ lies in $\Omega_\ell$.
For a squarefree integer $n$ divisible by $6$, define
\[
\mathcal{T}^n(N,\mathscr{R})
:=
\bigcap_{\ell\mid n}\mathcal{T}_\ell(N,\mathscr{R}),
\qquad
\mathcal{T}_{\op{cf}}(N,\mathscr{R})
:=
\bigcap_\ell \mathcal{T}_\ell(N,\mathscr{R}).
\]

Let $m_{i,j}(a_2,a_3, a_4)$ be the number of pairs $(\bar{a}_1,\bar{a}_5)\in (\Z/2^{6}3^{5})^2$ such that
\begin{itemize}
    \item $2^2\nmid \bar{a}_1\bar{a}_5,$
    \item $3^2\nmid \bar{a}_1\bar{a}_5$,
    \item $(\bar{a}_1, \bar{a}_2, \dots, \bar{a}_5)$ belongs to $\Omega_6$.
\end{itemize}

\begin{lemma}\label{Tn-asymp}
Let $n$ be squarefree, divisible by $6$, and by all primes $\ell\le \op{max}\{R_2,R_3\}$.
Then
\[
\begin{split}
\#\mathcal{T}^n(N,\mathscr{R})
=&
\frac{N}{2592}
\log\!\left(\frac{R_1}{R_1'}\right)
\prod_{\substack{\ell\mid n\\ \ell\neq 2,3}}
\left(1-\frac{1}{\ell^2}\right) \\
&\times
\sum_{\substack{a_2a_4\in[R_2',R_2]\\
a_3\in[R_3',R_3]\\
a_2a_4a_3\ \mathrm{squarefree}}}
\frac{m_{i,j}(a_2,a_3,a_4)
\prod_{\substack{\ell\mid a_2a_4a_3\\ \ell\neq 2,3}}
\left(\frac{\ell-1}{\ell+2}\right)}
{a_2^4 a_3^3 a_4^4}
+O\!\left(N^{1/2}\right),
\end{split}
\]
where the implied constant depends only on $\mathscr{R}$ and $n$.
\end{lemma}

\begin{proof}
As in the proof of Lemma \ref{lem:Rn-asymp}, 
\[\#\mathcal{T}^n(N,\mathscr{R})=\frac{N}{2 n^4}\log\!\left(\frac{R_1}{R_1'}\right)\sum_{a_2a_4\in [R_2',R_2]}\sum_{a_3\in [R_3',R_3]}\frac{\#\{z\in \widetilde{\Omega}_n\mid z_i\equiv a_i\pmod{N(n)}\text{ for }i=2,3,4\}}{a_2^4a_3^3a_4^4}+O(N^{1/2}).\]
Note that 
\[\begin{split}
    & \sum_{a_3\in [R_2',R_2]}\sum_{a_2a_4\in [R_3', R_3]}\frac{\#\{z\in \widetilde{\Omega}_n\mid z_i\equiv a_i\pmod{N(n)}\text{ for }i=2,3,4\}}{a_2^4 a_3^3 a_4^4}\\
    =& \sum_{\substack{a_2a_4\in [R_3', R_3]\\
    a_2a_4\text{ is squarefree}}}\frac{m_{i,j}(a_2,a_3,a_4)\prod_{\substack{\ell|n, \ell\nmid a_2 a_3 a_4\\ \ell\neq 2,3}}\ell^{4}\left(1-\frac{1}{\ell}\right)\left(1+\frac{1}{\ell}\right)\times \prod_{\substack{ \ell\mid a_2a_3a_4\\ \ell\neq 2,3}}\ell^4(1-\frac{1}{\ell})^2}{a_2^4 a_3^3 a_4^4}\\
    =& (n/6)^4\prod_{\substack{\ell|n\\\ell\neq 2,3}}\left(1-\frac{1}{\ell^2}\right)\sum_{\substack{a_2a_4\in [R_3', R_3]\\
    a_2a_4\text{ is squarefree}}}\frac{m_{i,j} (a_2,a_3,a_4)\prod_{\substack{\ell|a_2a_3a_4\\\ell\neq 2,3}}\left(\frac{\ell-1}{\ell+1}\right)}{a_2^4 a_3^3 a_4^4}
\end{split}\]and the result follows.
\end{proof}

For a prime $\ell$, let $\mathcal{T}_\ell'(N,\mathscr{R})$ be the complement of
$\mathcal{T}_\ell(N,\mathscr{R})$.

\begin{lemma}\label{Tell-tail}
With respect to notation above, one has that
\[
\frac{\#\mathcal{T}_\ell'(N,\mathscr{R})}{N}
=
O\!\left(\ell^{-6}+N^{-1/2}\right).
\]
\end{lemma}

\begin{proof}
The argument is almost identical to that of Lemma \ref{Rellprime lemma} and is therefore omitted.
\end{proof}

Let $\mathcal{E}_Y(N,\mathscr{R}):=\bigcup_{\ell>Y}\mathcal{T}_\ell'(N,\mathscr{R})$.

\begin{lemma}\label{T-tail}
\[
\frac{\#\mathcal{E}_Y(N,\mathscr{R})}{N}
=
O\!\left(\sum_{\ell>Y}\ell^{-6}+\frac{1}{N^{2/5}\log N}\right).
\]
\end{lemma}

\begin{proof}
Arguing as in Lemma~\ref{tail estimate}, we find that $\mathcal{T}_\ell'(N,\mathscr{R})=\emptyset$ unless $\ell\ll N^{1/10}$. Thus Lemma \ref{T-tail} implies that \[\frac{\# \mathcal{D}_Y(N,\mathscr{R})}{N}=O\left(\sum_{\ell>Y} \ell^{-6}+ \frac{\pi(N^{1/10})}{N^{1/2}}\right),\]
    and the result follows from the Prime number theorem.
\end{proof}

\begin{proposition}\label{T-main}
One has
\[
\begin{split}
\lim_{N\to\infty}
\frac{\#\mathcal{T}_{\op{cf}}(N,\mathscr{R})}{N}
=&
\frac{1}{2592}
\log\!\left(\frac{R_1}{R_1'}\right)
\prod_{\substack{\ell\neq 2,3}}
\left(1-\frac{1}{\ell^2}\right) \\
&\times
\sum_{\substack{a_2a_4\in[R_2',R_2]\\
a_3\in[R_3',R_3]\\
a_2a_4a_3\ \mathrm{squarefree}}}
\frac{m_{i,j}(a_2,a_3,a_4)
\prod_{\substack{\ell\mid a_2a_4a_3\\ \ell\neq 2,3}}
\left(\frac{\ell-1}{\ell+2}\right)}
{a_2^4 a_3^3 a_4^4}.
\end{split}
\]
\end{proposition}
\begin{proof}
    The proof of this result follows from Lemma \ref{Tn-asymp} and Lemma \ref{T-tail} via an argument identical to that in the proof of Proposition \ref{main propn}. 
\end{proof}

\par In light of Proposition \ref{T-main}, we are now in a position to state and prove the main result of this section. Let $\mathscr{R}:=\prod_{i=1}^3 [R_i', R_i]$ and for positive integers $(m,n)$, set \[\beta(m, n):=\prod_{\substack{\ell|mn\\\ell\neq 2,3}} \frac{\ell-1}{\ell+1} \sum_{n_1 n_2=n} \frac{m_{i,j}(n_1, m,n_2)}{n^4m^3}\] for $mn$ squarefree and $0$ otherwise. Define \[\nu_\beta(x,y)dxdy:=\sum_{\substack{m, n\geq 1\\ mn \text{ squarefree}}} \beta(m,n)\delta_{(m,n)}(x,y)\]
where $\delta_{(m,n)}$ is the Dirac measure on $\mathbb{R}^2$ supported at the point $(m,n)$. Then let $\hat{\nu}_{i,j}$ be the measure defined by the product:
\begin{equation}\label{nu i, j defn}\hat{\nu}_{i,j}:=\frac{\prod_{\substack{\ell\neq 2,3}}\left(1-\frac{1}{\ell^2}\right)}{2592}x_1^{-1}\nu_\beta(x_2,x_3) dx_1dx_2dx_3.\end{equation}

\begin{theorem}\label{main thm 2}
    With respect to notation above, we have that 
    \[\lim_{N\rightarrow \infty} \frac{\#\{K\mid m\text{ is of sign }\epsilon \text{ of Type }(i,j)\text{ with }(a_1/a_5,\, a_3,\, a_2a_4)\in \mathscr{R}\text{, and }|\Delta_K|\leq 2^{v_1}3^{v_2}N\}}{N^{1/5}}=\int_{\mathscr{R}} \hat{\nu}_{i,j}.\]
\end{theorem}
\begin{proof}
    As in the proof of Theorem \ref{main thm}, this result follows directly from Proposition \ref{T-main}.
\end{proof}
\bibliographystyle{alpha}
\bibliography{references}

@misc{Minkowski,
 author = {Minkowski, Hermann},
 title = {Gesammelte {Abhandlungen} von {Hermann} {Minkowski}. {Unter} {Mitwirkung} von {Andreas} {Speiser} und {Hermann} {Weyl}, herausgegeben von {David} {Hilbert}. {Band} {I}, {II}},
 year = {1911},
 language = {German},
}

@article{Grenier,
 author = {Grenier, Douglas},
 title = {Fundamental domains for the general linear group},
 fjournal = {Pacific Journal of Mathematics},
 journal = {Pac. J. Math.},
 issn = {1945-5844},
 volume = {132},
 number = {2},
 pages = {293--317},
 year = {1988},
}

@article {J21,
    AUTHOR = {Jakhar, Anuj},
     TITLE = {Explicit integral basis of pure sextic number fields},
   JOURNAL = {Rocky Mountain J. Math.},
  FJOURNAL = {Rocky Mountain Mathematics Consortium},
    VOLUME = {51},
      YEAR = {2021},
    NUMBER = {2},
     PAGES = {571--580},
}

@article {JKSMathematika,
    AUTHOR = {Jakhar, Anuj and Khanduja, Sudesh K. and Sangwan, Neeraj},
     TITLE = {On integral basis of pure number fields},
   JOURNAL = {Mathematika},
  FJOURNAL = {Mathematika. A Journal of Pure and Applied Mathematics},
    VOLUME = {67},
      YEAR = {2021},
    NUMBER = {1},
     PAGES = {187--195},
}

@book {DeloneFaddeev,
    AUTHOR = {Delone, B. N. and Faddeev, D. K.},
     TITLE = {The theory of irrationalities of the third degree},
    SERIES = {Translations of Mathematical Monographs, Vol. 10},
 PUBLISHER = {American Mathematical Society, Providence, RI},
      YEAR = {1964},
     PAGES = {xvi+509},
   MRCLASS = {10.25 (10.00)},
  MRNUMBER = {160744},
}

@article {HighercompositionIV,
    AUTHOR = {Bhargava, Manjul},
     TITLE = {Higher composition laws. {IV}. {T}he parametrization of
              quintic rings},
   JOURNAL = {Ann. of Math. (2)},
  FJOURNAL = {Annals of Mathematics. Second Series},
    VOLUME = {167},
      YEAR = {2008},
    NUMBER = {1},
     PAGES = {53--94},
}

@article {HighercompositionIII,
    AUTHOR = {Bhargava, Manjul},
     TITLE = {Higher composition laws. {III}. {T}he parametrization of
              quartic rings},
   JOURNAL = {Ann. of Math. (2)},
  FJOURNAL = {Annals of Mathematics. Second Series},
    VOLUME = {159},
      YEAR = {2004},
    NUMBER = {3},
     PAGES = {1329--1360},
      ISSN = {0003-486X},
}

@book {Terr97,
    AUTHOR = {Terr, David Charles},
     TITLE = {The distribution of shapes of cubic orders},
      NOTE = {Thesis (Ph.D.)--University of California, Berkeley},
 PUBLISHER = {ProQuest LLC, Ann Arbor, MI},
      YEAR = {1997},
     PAGES = {137},
}

@article {BS14,
    AUTHOR = {Bhargava, Manjul and Shnidman, Ariel},
     TITLE = {On the number of cubic orders of bounded discriminant having
              automorphism group {$C_3$}, and related problems},
   JOURNAL = {Algebra Number Theory},
  FJOURNAL = {Algebra \& Number Theory},
    VOLUME = {8},
      YEAR = {2014},
    NUMBER = {1},
     PAGES = {53--88},
}

@article {BH16,
    AUTHOR = {Bhargava, Manjul and Harron, Piper},
     TITLE = {The equidistribution of lattice shapes of rings of integers in
              cubic, quartic, and quintic number fields},
   JOURNAL = {Compos. Math.},
  FJOURNAL = {Compositio Mathematica},
    VOLUME = {152},
      YEAR = {2016},
    NUMBER = {6},
     PAGES = {1111--1120},
}

@article {MSM16,
    AUTHOR = {Mantilla-Soler, Guillermo and Monsurr\`o, Marina},
     TITLE = {The shape of {$\Bbb{Z}/\ell\Bbb{Z}$}-number fields},
   JOURNAL = {Ramanujan J.},
  FJOURNAL = {Ramanujan Journal. An International Journal Devoted to the
              Areas of Mathematics Influenced by Ramanujan},
    VOLUME = {39},
      YEAR = {2016},
    NUMBER = {3},
     PAGES = {451--463},
}

@article {Har17,
    AUTHOR = {Harron, Robert},
     TITLE = {The shapes of pure cubic fields},
   JOURNAL = {Proc. Amer. Math. Soc.},
  FJOURNAL = {Proceedings of the American Mathematical Society},
    VOLUME = {145},
      YEAR = {2017},
    NUMBER = {2},
     PAGES = {509--524},
}

@article {Hol22,
    AUTHOR = {Holmes, Erik},
     TITLE = {On the shapes of pure prime-degree number fields},
   JOURNAL = {J. Th\'eor. Nombres Bordeaux},
  FJOURNAL = {Journal de Th\'eorie des Nombres de Bordeaux},
    VOLUME = {37},
      YEAR = {2025},
    NUMBER = {1},
     PAGES = {1--48},
      ISSN = {1246-7405,2118-8572},
   MRCLASS = {11R45 (11E12 11P21 11R21)},
  MRNUMBER = {4932441},
}

@misc{purequartic,
 author = {Das, Sudipa and Kala, Sushant and Mukhopadhyay, Arunabha and Ray, Anwesh},
 title = {On the distribution of shapes of pure quartic number fields},
 year = {2025},
 howpublished = {Preprint, {arXiv}:2506.23766 [math.{NT}] (2025)},
 url = {https://arxiv.org/abs/2506.23766},
 arXiv = {arXiv:2506.23766}
}
\end{document}